\documentclass[11pt]{article}
\usepackage{amsfonts}
\usepackage{mathrsfs}
\usepackage{amsthm}
\usepackage{amssymb}
\usepackage{amsmath}
\usepackage{enumerate}
\usepackage{pb-diagram}

\theoremstyle{plain}
\newtheorem{theorem}{Theorem}[section]
\newtheorem{proposition}[theorem]{Proposition}
\newtheorem{corollary}[theorem]{Corollary}
\newtheorem{lemma}[theorem]{Lemma}

\theoremstyle{definition}
\newtheorem{definition}{Definition}[section]

\theoremstyle{remark}
\newtheorem{remark}{Remark}[section]

\theoremstyle{example}
\newtheorem{example}{Example}[section]

\numberwithin{equation}{section}

\title{Weighted Number Operators on Bernoulli Functionals and Quantum Exclusion Semigroups}
\author{Caishi Wang,\ \ Yuling Tang,\ \ Suling Ren\\
          School of Mathematics and Statistics\\
          Northwest Normal University,
          Lanzhou, Gansu 730070}
\date{}
\begin{document}
\maketitle

\noindent\textbf{Abstract}\ \
Quantum Bernoulli noises (QBN, for short) are the family of annihilation and creation operators acting on Bernoulli functionals,
which satisfy a canonical anti-commutation relation (CAR) in equal-time. In this paper, by using QBN, we first introduce a class of
self-adjoint operators acting on Bernoulli functionals, which we call the weighted number operators. We then make clear spectral
decompositions of these operators, and establish their commutation relations with the annihilation as well as the creation operators.
We also obtain a necessary and sufficient condition for a weighted number operator to be bounded.
Finally, as application of the above results, we construct a class of quantum Markov semigroups associated with the weighted number operators,
which belong to the category of quantum exclusion semigroups.
Some basic properties are shown of these quantum Markov semigroups, and examples are also given.

\vskip 2mm

\vskip 2mm


\section{Introduction}

Quantum Markov semigroups (QMS, for short) are quantum analogues of the classical Markov semigroups in the theory of probability, which provide
a mathematical model for describing the irreversible and loss-memory evolution of a quantum system interacting with
the environment, i.e., open quantum system (see, e.g. \cite{alicki, attal,meyer,partha}).
In the past four decades, general QMS have been studied extensively and many deep results have been obtained
(see, e.g. \cite{accardi,carlen, ch-fa, davies,fagnola,fag-reb-2} and references therein).
One typical result in this aspect is the theorem established by Gorini, Kossakowski, Sudershan and Lindblad, which
gives a characterization of the generater of a uniformly continuous QMS (see \cite{partha} for details).

As a special type of QMS, quantum exclusion semigroups can be viewed as quantum analogs of the classical exclusion semigroups
(exclusion processes) in the theory of probability \cite{liggett}.
Attention has been paid to quantum exclusion semigroups in recent years.
Pantale\'{o}n-Mart\'{i}nez and Quezada \cite{panta} constructed a quantum exclusion semigroup,
which they called the asymmetric exclusion quantum Markov semigroup, on the infinite tensor product space $\bigotimes_{l\in \mathbb{Z}^d}\mathbb{C}^2$.
In 2005, Rebolledo \cite{rebolledo} considered a class of quantum exclusion semigroups on the fermionic Fock space
and investigated their decoherence property.

Quantum Bernoulli noises are the family $\{\partial_k, \partial_k^*\mid k\geq 0\}$ of annihilation and creation operators acting
on Bernoulli functionals, which satisfy a canonical anti-commutation relation (CAR) in equal-time and play an active role in
building a discrete-time quantum stochastic calculus in infinite dimensions \cite{wcl,cw}.
In 2016, by using quantum Bernoulli noises, Wang and Chen \cite{wang-chen-2016} constructed a QMS with a formal generator of the following form
\begin{equation}\label{eq-1-1}
  \mathcal{L}(X)
   = \mathrm{i}[H,X] -\frac{1}{2}\sum_{k=0}^{\infty}
   \big(X\partial_k^*\partial_k -2\partial_k^*X\partial_k
   +\partial_k^*\partial_kX\big),
\end{equation}
where $H$ is a self-adjoint operator acting on Bernoulli functionals and $\{\partial_k, \partial_k^*\mid k\geq 0\}$ are quantum Bernoulli noises.
In the present paper, we aim to investigate quantum exclusion semigroups in terms of quantum Bernoulli noises.
More precisely, we would like to construct a QMS with a formal generator of the following form
\begin{equation}\label{eq-1-2}
\begin{split}
  \mathcal{L}(X)
    &= \mathrm{i}[H,X]\\
    &\quad  -\frac{1}{2}\sum_{j,k=0}^{\infty}w(j,k)\big[ X \big(\partial_j^*\partial_k\big)^*\partial_j^*\partial_k
           - 2 \big(\partial_j^*\partial_k\big)^* X\partial_j^*\partial_k + \big(\partial_j^*\partial_k\big)^*\partial_j^*\partial_kX\big],
\end{split}
\end{equation}
where $H$ is a self-adjoint operator acting on Bernoulli functionals, $w$ is some nonnegative function and $\{\partial_k, \partial_k^*\mid k\geq 0\}$
are quantum Bernoulli noises.
From a physical point of view, a QMS with a formal generator of form (\ref{eq-1-2})
belongs to the category of quantum exclusion semigroups, and might serve as a model describing
an open quantum system consisting of an arbitrary number of identical Fermi particles,
where $H$ represents the Hamiltonian of the system and $\partial_j^*\partial_k$ represents the jump of particles from site $k$ to site $j$ with
the jump rate $\sqrt{w(j,k)}$.

To construct a QMS with a formal generator of form (\ref{eq-1-2}), one needs to deal with
the operator series
\begin{equation}\label{eq-1-3}
  \sum_{j,k=0}^{\infty}w(j,k)\big(\partial_j^*\partial_k\big)^*\partial_j^*\partial_k
  =\sum_{j,k=0}^{\infty}w(j,k)\partial_k^*\partial_j\partial_j^*\partial_k,
\end{equation}
which acts on Bernoulli functionals. So, our another goal in this paper is to make clear the structure
of the sum of operator series (\ref{eq-1-3}), which itself is interesting from a mathematical point of view.

Our main work in this paper consists of two parts.
In the first part, we define in a natural way the sum $S_w$ of operator series (\ref{eq-1-3}), which we call the 2D-weighted number operator
associated with $w$. We obtain the spectral decomposition of $S_w$ and find out a necessary and sufficient condition for $S_w$ to be bounded.
We establish the commutation relations of $S_w$ with $\partial_k$ as well as $\partial_k^*$,
which differ from the well-known CCR and CAR.
We mention that we also introduce in this part a notion of the 1D-weighted number operators and show structure properties of these operators.
In the second part, we show the possibility to construct a QMS with a formal generator of form~(\ref{eq-1-2}).
This is done by combining the general results on QMS with our results
on the 2D-weighted number operators $S_w$, which are obtained in the first part.
An easily verifiable sufficient condition is found out for the existence of a QMS with a formal generator of form~(\ref{eq-1-2}).
And finally, an example is offered to show the effectiveness of this condition and some other results are also obtained.

The paper is organized as follows. In Section~\ref{sec-2} and Section~\ref{sec-3}, we briefly recall some necessary notions
and known results about a general QMS, and collect necessary fundamentals about
quantum Bernoulli noises. Our main work then lies in Section~\ref{sec-4} and Section~\ref{sec-5},
where Section~\ref{sec-4} is divided into two subsections for clarity.

{\bf General notation and conventions.} Throughout, $\mathbb{R}$ always denotes the set of all real numbers
and $\mathbb{N}$ the set of all nonnegative integers.
We denote by $\Gamma$ the finite power set of $\mathbb{N}$, namely
\begin{equation}\label{eq-1-4}
    \Gamma
    = \{\,\sigma \mid \text{$\sigma \subset \mathbb{N}$ and $\#(\sigma) < \infty$} \,\},
\end{equation}
where $\#(\sigma)$ means the cardinality of $\sigma$ as a set.
By $\Re z$ and $\Im z$ we mean the real part and imaginary part of a complex number $z$, respectively.
If $\langle\cdot,\cdot\rangle$ is a complex inner product,
then it is conjugate linear in its first variable and linear in its second one.
For an operator $A$, we denote by $\mathrm{Dom}\, A$ its domain.
If $A$ is a densely defined operator in a Hilbert space, then $A^*$ means its adjoint operator.

\section{General results on quantum Markov semigroup}\label{sec-2}

In this section, we briefly recall some necessary notions and general results about QMS.
We refer to \cite{ch-fa, fagnola} and references therein for more details.

Let $\mathsf{H}$ be a complex separable Hilbert space with inner product $\langle\cdot,\cdot\rangle$ and norm $\|\cdot\|$.
As usual, we denote by $I$ the identity operator on $\mathsf{H}$.

Let $\mathfrak{B}(\mathsf{H})$ be the Banach algebra of all bounded operators on $\mathsf{H}$ and $\|\cdot\|_{\infty}$ the norm
in $\mathfrak{B}(\mathsf{H})$. A linear map $\mathcal{M}\colon \mathfrak{B}(\mathsf{H}) \rightarrow \mathfrak{B}(\mathsf{H})$ is said
to be completely positive if for all integer $n\geq 1$ and all finite sequence $(X_i)_{i=1}^n$, $(Y_i)_{i=1}^n$ of elements of $\mathfrak{B}(\mathsf{H})$
it holds that
\begin{equation}\label{eq-2-1}
  \sum_{i,j=1}^n Y_i^*\mathcal{M}(X_i^*X_j)Y_j\geq 0,
\end{equation}
namely the sum is a positive operator on $\mathsf{H}$.

\begin{definition}\label{def-2-1}
A quantum dynamical semigroup (QDS, for short) on $\mathfrak{B}(\mathsf{H})$ is a one parameter family $\mathcal{T} =(\mathcal{T}_t)_{t\geq 0}$
of linear maps on $\mathfrak{B}(\mathsf{H})$ with the following properties:
\begin{enumerate}
  \item[(i)] $\mathcal{T}_0(X) = X$, for all $X\in \mathfrak{B}(\mathsf{H})$;
  \item[(ii)] $\mathcal{T}_{t+s}(X) = \mathcal{T}_t (\mathcal{T}_s(X))$, for all $s$, $t \geq 0$ and all $X\in \mathfrak{B}(\mathsf{H})$;
  \item[(iii)] $\mathcal{T}_t(I) \leq I$, for all $t \geq 0$;
  \item[(iv)] $\mathcal{T}_t\colon \mathfrak{B}(\mathsf{H}) \rightarrow \mathfrak{B}(\mathsf{H})$ is completely positive for all $t \geq 0$;
  \item[(v)] for every sequence $(X_n)_{n\geq 1}$ of elements of $\mathfrak{B}(\mathsf{H})$ converging weakly to an element $X$ of
             $\mathfrak{B}(\mathsf{H})$, the sequence $(\mathcal{T}_t(X_n))_{n\geq 1}$ converges weakly to $\mathcal{T}_t(X)$ for all $t \geq 0$;
  \item[(vi)] for all trace class operator $\rho$ on $\mathsf{H}$ and all $X\in \mathfrak{B}(\mathsf{H})$, it holds that
        \begin{equation*}
        \lim_{t\to 0^+} \mathrm{Tr}(\rho \mathcal{T}_t(X)) = \mathrm{Tr}(\rho X).
        \end{equation*}
\end{enumerate}
\end{definition}

We note that, as a consequence of properties (iii) and (iv), one has the inequality
\begin{equation}\label{eq-2-2}
  \|\mathcal{T}_t(X)\|_{\infty} \leq \|X\|_{\infty}
\end{equation}
for all $X\in \mathfrak{B}(\mathsf{H})$ and all $t\geq 0$. Thus, for all $t\geq 0$, $\mathcal{T}_t$ is a
bounded linear (hence continuous) map on $\mathfrak{B}(\mathsf{H})$. We also note that, as a consequence of properties (iv) and (vi),
the map $t\rightarrow \mathcal{T}_t(X)$ is continuous with respect to the strong operator topology for all $X\in \mathfrak{B}(\mathsf{H})$.

\begin{definition}\label{def-2-2}
A QDS $\mathcal{T} =(\mathcal{T}_t)_{t\geq 0}$ on $\mathfrak{B}(\mathsf{H})$ is said to be
conservative if $\mathcal{T}_{t }(I) = I$, for every $t\geq 0$. A conservative QDS is called a quantum Markov semigroup
(QMS, for short).
\end{definition}

In the mathematical literature,  a QMS is also known as an identity-preserving quantum dynamical semigroup. The next lemma shows the
existence of a QDS with unbounded generator.

\begin{lemma}\label{lem-2-1}\cite{ch-fa}
Let $G$ be the infinitesimal generator of a strongly continuous contraction semigroup $P=(P_t)_{t \geq 0}$ on $\mathsf{H}$,
and $(L_k)_{k\geq1}$ a sequence of operators defined in $\mathsf{H}$ and with the property that for all $k\geq 1$
the domain of $L_k$ contains the domain of $G$. Assume that
\begin{equation}\label{eq-2-3}
  \langle u, Gv\rangle + \langle Gu, v\rangle + \sum_{k=1}^{\infty}\langle L_ku, L_k v\rangle =0,\quad \forall\, u, v \in \mathrm{Dom}\, G.
\end{equation}
Then there exists a minimal QDS $\mathcal{T} =(\mathcal{T}_t)_{t\geq 0}$ on $\mathfrak{B}(\mathsf{H})$
satisfying the equation
\begin{equation}\label{eq-2-4}
\begin{split}
 &\langle u,\mathcal{T}_t(X) v\rangle\\
 & =\langle u,X v\rangle
  +\int_0^t\Big[\langle u,\mathcal{T}_s(X) Gv\rangle + \langle Gu, \mathcal{T}_s(X)v\rangle
  + \sum_{k=1}^{\infty}\langle L_ku, \mathcal{T}_s(X) L_k v\rangle\Big]ds
 \end{split}
\end{equation}
for all $u$, $v \in \mathrm{Dom}\, G$ and all $X\in \mathfrak{B}(\mathsf{H})$.
\end{lemma}

\begin{remark}
Here in Lemma~\ref{lem-2-1} the name ``minimal'' means that if $\mathcal{S} =(\mathcal{S}_t)_{t\geq 0}$ is
another QDS on $\mathfrak{B}(\mathsf{H})$ satisfying
\begin{equation*}
\begin{split}
 &\langle u,\mathcal{S}_t(X) v\rangle\\
 & =\langle u,X v\rangle
  +\int_0^t\Big[\langle u,\mathcal{S}_s(X) Gv\rangle + \langle Gu, \mathcal{S}_s(X)v\rangle
  + \sum_{k=1}^{\infty}\langle L_ku, \mathcal{S}_s(X) L_k v\rangle\Big]ds
 \end{split}
\end{equation*}
for all $u$, $v \in \mathrm{Dom}\, G$ and all $X\in \mathfrak{B}(\mathsf{H})$, then it must satisfy that
\begin{equation}\label{eq-2-5}
  \mathcal{T}_t(X)\leq \mathcal{S}_t(X) \leq \|X\|_{\infty}I
\end{equation}
for all positive elements $X$ of $\mathfrak{B}(\mathsf{H})$ and all $t\geq 0$.
\end{remark}

\begin{definition}
Let $G$ and $(L_k)_{k\geq1}$ be the same as in Lemma~\ref{lem-2-1}. Then the minimal QDS
$\mathcal{T} =(\mathcal{T}_t)_{t\geq 0}$ on $\mathfrak{B}(\mathsf{H})$
satisfying (\ref{eq-2-4}) is called the minimal QDS constructed from $G$ and $(L_k)_{k\geq1}$.
\end{definition}

The next lemma is actually Corollary~4.5 of \cite{ch-fa}, which provides a simple and easily verifiable
condition for the minimal QDS to be conservative, namely to be a QMS.

\begin{lemma}\label{lem-2-2}\cite{ch-fa}
Let $G$ and $(L_k)_{k\geq1}$ be the same as in Lemma~\ref{lem-2-1}. Assume further that there exists a self-adjoint operator $C$
in $\mathsf{H}$ and a core $\mathcal{D}$ for $C$ satisfying that:
\begin{enumerate}
  \item[(a)] the domain of $C$ coincides with the domain of $G$, and for all $u \in \mathrm{Dom}\, C$, there exists a sequence
  $(u_n)_{n\geq 1}\subset \mathcal{D}$ such that both $(Gu_n)_{n\geq 1}$ and $(Cu_n)_{n\geq 1}$ converge in $\mathsf{H}$;
  \item[(b)] there exists a positive self-adjoint operator $\Phi$ in $\mathsf{H}$ such that the domain of $\Phi$ contains $\mathcal{D}$ and
  \begin{equation}\label{eq-2-6}
  -2\Re\langle u, Gu\rangle = \langle u, \Phi u \rangle\leq \langle u, Cu\rangle,\quad \forall\, u\in \mathcal{D};
  \end{equation}
  \item[(c)] for all $k\geq 1$, $L_k(\mathcal{D})\subset \mathrm{Dom}\, C$;
  \item[(d)] there exists a constant $b$ such that, for all $u\in \mathcal{D}$, the following inequality holds
  \begin{equation}\label{eq-2-7}
    2\Re\langle Cu, Gu\rangle + \sum_{k=1}^{\infty}\langle L_ku, C L_k u\rangle
    \leq b \langle u, Cu\rangle.
  \end{equation}
\end{enumerate}
Then the minimal QDS $\mathcal{T} =(\mathcal{T}_t)_{t\geq 0}$ constructed from $G$ and $(L_k)_{k\geq1}$ is conservative, namely a QMS.
\end{lemma}

\section{Quantum Bernoulli noises}\label{sec-3}

In this section, we briefly describe the main notions and facts about quantum Bernoulli noises (see, e.g. \cite{wcl,wang-chen-2016} for details).

In the following, we denote by $\Gamma$ the finite power set of $\mathbb{N}$, namely
\begin{equation}\label{eq-3-1}
    \Gamma
    = \{\,\sigma \mid \text{$\sigma \subset \mathbb{N}$ and $\#\,\sigma < \infty$} \,\},
\end{equation}
where $\#\sigma$ means the cardinality of $\sigma$ as a set.
Throughout, we assume that $(\Omega, \mathscr{F}, \mathbb{P})$ is a probability space and
$Z=(Z_n)_{n\geq 0}$ is an independent sequence of random variables on
$(\Omega, \mathscr{F}, \mathbb{P})$, which satisfies that
\begin{equation}\label{eq-3-2}
    \mathbb{P}\{Z_n = \theta_n\}=p_n,\quad
    \mathbb{P}\{Z_n = -1/\theta_n\}=q_n,\quad n\geq 0
\end{equation}
with $\theta_n = \sqrt{q_n/p_n}$, $q_n = 1-p_n$ and $0 < p_n < 1$ and, moreover,
$\mathscr{F}=(Z_n; n \geq 0)$, the $\sigma$-field generated by $Z=(Z_n)_{n\geq 0}$.

\begin{remark}\label{rem-3-1}
As is seen, $Z$ is actually a discrete-time Bernoulli stochastic process. Additionally, because $\mathscr{F}=(Z_n; n \geq 0)$,
complex random variables on $(\Omega, \mathscr{F}, \mathbb{P})$ are usually called
functionals of $Z$ (known as Bernoulli functionals in general).
\end{remark}

To be convenient, we set $\mathscr{F}_{-1}=\{\emptyset, \Omega\}$ and
$\mathscr{F}_n = \sigma(Z_k; 0\leq k \leq n)$,
the $\sigma$-field generated by $(Z_k)_{0\leq k \leq n}$,
for $ n \geq 0$.
By convention $\mathbb{E}$ will denote the expectation with respect to $\mathbb{P}$.

Let $\mathsf{L}\!^2(Z)$ be the space of square integrable
complex-valued random variables on $(\Omega, \mathscr{F}, \mathbb{P})$, namely
\begin{equation}\label{eq-3-3}
  \mathsf{L}\!^2(Z) = L^2(\Omega, \mathscr{F}, \mathbb{P}).
\end{equation}
We denote by
$\langle\cdot,\cdot\rangle$ the usual inner product of the space $\mathsf{L}\!^2(Z)$, which is defined by
 \begin{equation*}
   \langle \xi,\eta\rangle = \mathbb{E}[\overline{\xi}\eta],
\end{equation*}
and by $\|\cdot\|$ the corresponding norm.
It is known \cite{privault} that $Z$ has the chaotic representation property.
Thus $\mathsf{L}\!^2(Z)$ has $\{Z_{\sigma}\mid \sigma \in \Gamma\}$ as its orthonormal basis,
where $Z_{\emptyset}=1$ and
\begin{equation}\label{eq-3-4}
    Z_{\sigma} = \prod_{i\in \sigma}Z_i,\quad \text{$\sigma \in \Gamma$, $\sigma \neq \emptyset$},
\end{equation}
which shows that $\mathsf{L}\!^2(Z)$ is an infinite dimensional, separable complex Hilbert space.
In what follows, we call $\{Z_{\sigma}\mid \sigma \in \Gamma\}$ the canonical ONB of $\mathsf{L}\!^2(Z)$.

\begin{lemma}\label{lem-3-1}\cite{wcl}
For each $k\in \mathbb{N}$, there exists a bounded operator $\partial_k\colon
\mathsf{L}\!^2(Z)\rightarrow \mathsf{L}\!^2(Z)$
such that
\begin{equation}\label{eq-3-5}
    \partial_k Z_{\sigma} = \mathbf{1}_{\sigma}(k)Z_{\sigma\setminus k},\quad
    \partial_k^{\ast} Z_{\sigma}
    = [1-\mathbf{1}_{\sigma}(k)]Z_{\sigma\cup k},\quad
    \sigma \in \Gamma,
\end{equation}
where $\partial_k^{\ast}$ denotes the adjoint of $\partial_k$, $\sigma\setminus k=\sigma\setminus \{k\}$, $\sigma\cup k=\sigma\cup \{k\}$
and $\mathbf{1}_{\sigma}(k)$ the indicator of $\sigma$ as a subset of $\mathbb{N}$.
\end{lemma}

In the language of physics, the operator $\partial_k$ and its adjoint $\partial_k^{\ast}$ are referred to as the annihilation operator and
creation operator at site $k$, respectively.

\begin{definition}\label{def-3-1}\cite{wcl}
The family $\{\partial_k, \partial_k^{\ast}\}_{k \geq 0}$ of annihilation and creation operators
is called quantum Bernoulli noises (QBN, for short).
\end{definition}

The next lemma shows that QBN satisfy the canonical anti-commutation relations (CAR) in equal-time.

\begin{lemma}\label{lem-3-2}\cite{wcl}
Let $k$, $l\in \mathbb{N}$. Then it holds true that
\begin{equation}\label{eq-3-7}
    \partial_k \partial_l = \partial_l\partial_k,\quad
    \partial_k^{\ast} \partial_l^{\ast} = \partial_l^{\ast}\partial_k^{\ast},\quad
    \partial_k^{\ast} \partial_l = \partial_l\partial_k^{\ast}\quad (k\neq l)
\end{equation}
and
\begin{equation}\label{eq-3-8}
   \partial_k\partial_k= \partial_k^{\ast}\partial_k^{\ast}=0,\quad
   \partial_k\partial_k^{\ast} + \partial_k^{\ast}\partial_k=I,
\end{equation}
where $I$ is the identity operator on $\mathsf{L}\!^2(Z)$.
\end{lemma}

From Lemma~\ref{lem-3-1} and Lemma~\ref{lem-3-2}, one can easily get the following useful result.

\begin{lemma}\label{lem-3-3}
Let $k\in \mathbb{N}$. Then $\partial_k^*\partial_k$ is projection operator on $\mathsf{L}\!^2(Z)$, and moreover
\begin{equation}\label{eq-3}
  \partial_k^{\ast}\partial_k Z_{\sigma}
    =\mathbf{1}_{\sigma}(k)Z_{\sigma},\quad
    \sigma \in \Gamma.
\end{equation}
\end{lemma}

The number operator $N$ in $\mathsf{L}\!^2(Z)$ can be defined in many different ways. The following definition is taken from \cite{wang-chen-2016}.

\begin{definition}\label{def-3-2}
The number operator $N$ in $\mathsf{L}\!^2(Z)$ is the one given by
\begin{equation}\label{eq-3-9}
  N\xi = \sum_{\sigma \in \Gamma} \#(\sigma) \langle Z_{\sigma},\xi\rangle Z_{\sigma},\quad \xi \in \mathrm{Dom}\,N
\end{equation}
with
\begin{equation}\label{eq-3-10}
\mathrm{Dom}\,N =\Big\{\, \xi \in \mathsf{L}\!^2(Z) \Bigm| \sum_{\sigma \in \Gamma} \big(\#(\sigma)\big )^2 |\langle Z_{\sigma},\xi\rangle |^2 <\infty \,\Big\},
\end{equation}
where $\#(\sigma)$ means the cardinality of $\sigma$ as a set.
\end{definition}

Clearly, $N$ is  densely-defined, positive and self-adjoint, and moreover $\mathcal{D}\subset \mathrm{Dom}\,N$,
where $\mathcal{D}=\mathrm{span}\,\{Z_{\sigma} \mid \sigma\in \Gamma\}$
denotes the linear subspace (manifold) spanned by the canonical ONB of $\mathsf{L}\!^2(Z)$.

\section{Weighted number operator}\label{sec-4}

In this section, we introduce a class of operators densely defined in the space $\textsf{L}\!^2(Z)$,
which we call the 2D-weighted number operators.
Among others, we examine structures of these operators, and establish their commutation relations with the annihilation as well as the creation operators.

\subsection{Definition and structure}

We first clarify the meaning of convergence of a double vector series in $\textsf{L}\!^2(Z)$.

A double vector series $\sum_{j,k=0}^{\infty}\xi_{jk}$ in $\textsf{L}\!^2(Z)$ is said to converge naturally
if its sequence of square-array partial sums
\begin{equation*}
   s_n = \sum_{j,k=0}^n\xi_{jk}= \sum_{j=0}^n\sum_{k=0}^n\xi_{jk},\quad n\geq 0
\end{equation*}
converges in $\textsf{L}\!^2(Z)$ as $n\rightarrow \infty$. In that case, we write $\sum_{j,k=0}^{\infty}\xi_{jk}= \lim_{n\to \infty}s_n$.

\begin{proposition}\label{prop-4-1}
Let $w\colon \mathbb{N}\times \mathbb{N}\rightarrow \mathbb{R}$ be a nonnegative function.
Then, for all $\sigma\in \Gamma$ and all $n\geq 1$, it holds true that
\begin{equation}\label{eq-4-1}
\begin{split}
   \sum_{j,k=0}^n & w(j,k)\partial_k^*\partial_j\partial_j^*\partial_kZ_{\sigma}\\
      & = \Big[\sum_{j=0}^n \mathbf{1}_{\sigma}(j)w(j,j)
       + \sum_{j,k=0}^n \big(1-\mathbf{1}_{\sigma}(j)\big) \mathbf{1}_{\sigma}(k)w(j,k)\Big]Z_{\sigma}.
\end{split}
\end{equation}
Moreover, if the function $w$ satisfies that $\sup_{k\geq 0}\sum_{j=0}^{\infty} w(j,k)<\infty$, then the double vector series
\begin{equation*}
  \sum_{j,k=0}^{\infty} w(j,k)\partial_k^*\partial_j\partial_j^*\partial_kZ_{\sigma}
\end{equation*}
converges naturally for each $\sigma \in \Gamma$.
\end{proposition}

\begin{proof}
Let $\sigma \in \Gamma$, $n\geq 1$. Then, by CAR in equal time as well as properties of the product operator $\partial_n^*\partial_n$, we have
\begin{equation*}
\begin{split}
   \sum_{j,k=0}^n &w(j,k)\partial_k^*\partial_j\partial_j^*\partial_kZ_{\sigma}\\
   & = \sum_{j=0}^n\Big[w(j,j)\partial_j^*\partial_jZ_{\sigma}  + \sum_{k=0,k\neq j}^n w(j,k)\big(I-\partial_j^*\partial_j\big)\partial_k^*\partial_kZ_{\sigma}\Big]\\
   & = \sum_{j=0}^n\Big[\mathbf{1}_{\sigma}(j)w(j,j)Z_{\sigma}  + \sum_{k=0,k\neq j}^n \big(1-\mathbf{1}_{\sigma}(j)\big)\mathbf{1}_{\sigma}(k)w(j,k)Z_{\sigma}\Big]\\
   & = \sum_{j=0}^n\Big[\mathbf{1}_{\sigma}(j)w(j,j)Z_{\sigma}  + \sum_{k=0}^n \big(1-\mathbf{1}_{\sigma}(j)\big)\mathbf{1}_{\sigma}(k)w(j,k)Z_{\sigma}\Big]\\
   & = \Big[\sum_{j=0}^n\mathbf{1}_{\sigma}(j)w(j,j) + \sum_{j,k=0}^n \big(1-\mathbf{1}_{\sigma}(j)\big)\mathbf{1}_{\sigma}(k)w(j,k)\Big]Z_{\sigma}.
\end{split}
\end{equation*}
Now let the function $w$ satisfy that $\sup_{k\geq 0}\sum_{j=0}^{\infty} w(j,k)<\infty$. Then, by a calculation,  we find
\begin{equation*}
\begin{split}
  &\sum_{j=0}^{\infty}\mathbf{1}_{\sigma}(j)w(j,j) + \sum_{j,k=0}^{\infty} \big(1-\mathbf{1}_{\sigma}(j)\big)\mathbf{1}_{\sigma}(k)w(j,k)\\
  & = \sum_{j=0}^{\infty}\mathbf{1}_{\sigma}(j)w(j,j) + \sum_{k=0}^{\infty} \mathbf{1}_{\sigma}(k)\sum_{j}^{\infty}\big(1-\mathbf{1}_{\sigma}(j)\big)w(j,k)\\
  &\leq \sum_{j=0}^{\infty}\mathbf{1}_{\sigma}(j)\alpha + \sum_{k=0}^{\infty} \mathbf{1}_{\sigma}(k)\sum_{j=0}^{\infty}w(j,k)\\
  &\leq \sum_{j=0}^{\infty}\mathbf{1}_{\sigma}(j)\alpha + \sum_{k=0}^{\infty} \mathbf{1}_{\sigma}(k)\alpha\\
  & = 2\alpha\#(\sigma)\\
  &<\infty,
\end{split}
\end{equation*}
where $\alpha = \sup_{k\geq 0}\sum_{j=0}^{\infty}w(j,k)$,
which together with the fact
\begin{equation*}
\begin{split}
  &\lim_{n\to \infty}\Big[\sum_{j=0}^n  \mathbf{1}_{\sigma}(j)w(j,j)  + \sum_{j,k=0}^n \big(1-\mathbf{1}_{\sigma}(j)\big)\mathbf{1}_{\sigma}(k)w(j,k)\Big]\\
  & =   \sum_{j=0}^{\infty}\mathbf{1}_{\sigma}(j)w(j,j) + \sum_{j,k=0}^{\infty} \big(1-\mathbf{1}_{\sigma}(j)\big)\mathbf{1}_{\sigma}(k)w(j,k)
\end{split}
\end{equation*}
implies that
\begin{equation*}
  \sum_{j,k=0}^n w(j,k)\partial_k^*\partial_j\partial_j^*\partial_kZ_{\sigma}
  = \Big[\sum_{j=0}^n\mathbf{1}_{\sigma}(j)w(j,j) + \sum_{j,k=0}^n \big(1-\mathbf{1}_{\sigma}(j)\big)\mathbf{1}_{\sigma}(k)w(j,k)\Big]Z_{\sigma}
\end{equation*}
converges as $n\rightarrow \infty$, namely the double vector series
\begin{equation*}
  \sum_{j,k=0}^{\infty} w(j,k)\partial_k^*\partial_j\partial_j^*\partial_kZ_{\sigma}
\end{equation*}
converges naturally.
\end{proof}

\begin{definition}\label{def-4-1}
Let $w\colon \mathbb{N}\times \mathbb{N}\rightarrow \mathbb{R}$ be a nonnegative function with the property that
\begin{equation}\label{eq-4-2}
\sup_{k\geq 0}\sum_{j=0}^{\infty} w(j,k)<\infty.
\end{equation}
Then the 2D-weighted number operator $S_w$ associated with $w$ is the one
in $\textsf{L}\!^2(Z)$ given by
\begin{equation}\label{eq-4-3}
  S_w\xi=\sum_{j,k=0}^{\infty}w(j,k)\partial_k^*\partial_j\partial_j^*\partial_k\xi,\quad \xi \in \mathrm{Dom}\, S_w
\end{equation}
with
\begin{equation}\label{eq-4-4}
  \mathrm{Dom}\, S_w =\Big\{\, \xi \in \textsf{L}\!^2(Z) \Bigm| \sum_{j,k=0}^{\infty}w(j,k)\partial_k^*\partial_j\partial_j^*\partial_k\xi\ \ \mbox{converges naturally}               \,\Big\}.
\end{equation}
\end{definition}

Recall that the system $\{Z_{\sigma} \mid \sigma\in \Gamma\}$ is an orthonormal basis of $\textsf{L}\!^2(Z)$, which is called the canonical ONB
of $\textsf{L}\!^2(Z)$.
By Proposition~\ref{prop-4-1}, we find that $\{Z_{\sigma}\mid \sigma\in \Gamma\} \subset \mathrm{Dom}\,S_w$, which implies that
$\mathcal{D}=\mathrm{span}\, \{Z_{\sigma}\mid \sigma\in \Gamma\} \subset \mathrm{Dom}\,S_w$, hence
$S_w$ is densely defined in $\mathsf{L}\!^2(Z)$.

In what follows, when we say a 2D-weighted number operator we mean the 2D-weighted number operator associated with a nonnegative function $w$ satisfying
condition~(\ref{eq-4-2}).

The next theorem offers a characterization of the 2D-weighted number operators,
which shows that a 2D-weighted number operator admits a spectral decomposition in terms of the canonical ONB of $\textsf{L}\!^2(Z)$.

\begin{theorem}\label{thr-4-2}
Let $w\colon \mathbb{N}\times \mathbb{N}\rightarrow \mathbb{R}$ be a nonnegative function satisfying condition~(\ref{eq-4-2}).
Then a vector $\xi \in \mathsf{L}\!^2(Z)$ belongs to $\mathrm{Dom}\, S_w$ if and only if it satisfies that
\begin{equation}\label{eq-4-5}
  \sum_{\sigma \in \Gamma}\big(\vartheta_w(\sigma)\big)^2|\langle Z_{\sigma},\xi\rangle|^2<\infty,
\end{equation}
where $\vartheta_w$ is the function on $\Gamma$ given by
\begin{equation}\label{eq-4-6}
\vartheta_w(\sigma)
= \sum_{j=0}^{\infty}\mathbf{1}_{\sigma}(j)w(j,j) + \sum_{j,k=0}^{\infty} \big(1-\mathbf{1}_{\sigma}(j)\big)\mathbf{1}_{\sigma}(k)w(j,k),\quad\sigma \in \Gamma.
\end{equation}
In that case, one has
\begin{equation}\label{eq-4-7}
  S_w\xi = \sum_{\sigma \in \Gamma}\vartheta_w(\sigma)\langle Z_{\sigma},\xi\rangle Z_{\sigma}.
\end{equation}
\end{theorem}

\begin{proof}
First let $\xi \in \textsf{L}^2(Z)$ with
$\sum_{\sigma \in \Gamma}\big(\vartheta_w(\sigma)\big)^2|\langle Z_{\sigma},\xi\rangle|^2<\infty$.
We show that $\xi \in \mathrm{Dom}\, S_w$ and (\ref{eq-4-7}) holds.
To do so, we set
\begin{equation*}
\eta= \sum_{\sigma \in \Gamma}\vartheta_w(\sigma)\langle Z_{\sigma},\xi\rangle Z_{\sigma}
\end{equation*}
and
\begin{equation*}
   \eta_n = \sum_{j,k=0}^nw(j,k)\partial_k^*\partial_j\partial_j^*\partial_k\xi,\quad n\geq 0.
\end{equation*}
Then, for each $n\geq 0$, it follows from the boundedness of operators $\partial_k^*\partial_j\partial_j^*\partial_k$
as well as the norm convergence of the expansion
$\xi = \sum_{\sigma \in \Gamma} \langle Z_{\sigma}, \xi\rangle Z_{\sigma}$
that
\begin{equation*}
  \eta_n=\sum_{\sigma \in \Gamma}\langle Z_{\sigma}, \xi\rangle\sum_{j,k=0}^nw(j,k)  \partial_k^*\partial_j\partial_j^*\partial_kZ_{\sigma},
\end{equation*}
which together with Proposition~\ref{prop-4-1} gives
\begin{equation}\label{eq-4-8}
  \eta_n=\sum_{\sigma \in \Gamma} \vartheta_w^{(n)}(\sigma)\langle Z_{\sigma}, \xi\rangle Z_{\sigma},
\end{equation}
where
\begin{equation}\label{eq-4-9}
  \vartheta_w^{(n)}(\sigma)
  = \sum_{j=0}^n \mathbf{1}_{\sigma}(j)w(j,j)
       + \sum_{j,k=0}^n \big(1-\mathbf{1}_{\sigma}(j)\big) \mathbf{1}_{\sigma}(k)w(j,k).
\end{equation}
Thus, by the definition of $\eta$, we have
\begin{equation*}
\|\eta - \eta_n\|^2 = \sum_{\sigma \in \Gamma}|\vartheta_w(\sigma)-\vartheta_w^{(n)}(\sigma)|^2|\langle Z_{\sigma}, \xi\rangle|^2,\quad n\geq 0.
\end{equation*}
From (\ref{eq-4-9}), we find that $0\leq \vartheta_w^{(n)}(\sigma) \leq \vartheta_w(\sigma)$ for all $\sigma \in \Gamma$ and $n\geq 1$, which implies that
\begin{equation*}
  |\vartheta_w(\sigma)-\vartheta_w^{(n)}(\sigma)|^2
  \leq \big(\vartheta_w(\sigma) + \vartheta_w^{(n)}(\sigma)\big)^2
  \leq 4 \big(\vartheta_w(\sigma)\big)^2,\quad \sigma \in \Gamma ,\, n\geq 0.
\end{equation*}
Note that the function $\sigma \mapsto 4\big(\vartheta_w(\sigma)\big)^2$ satisfies that
\begin{equation*}
\sum_{\sigma \in \Gamma}4\big(\vartheta_w(\sigma)\big)^2|\langle Z_{\sigma},\xi\rangle|^2<\infty.
\end{equation*}
On the other hand, we easily see that $\vartheta_w^{(n)}(\sigma)\rightarrow \vartheta_w(\sigma)$ for all $\sigma\in \Gamma$,
which implies that
\begin{equation*}
  |\vartheta_w(\sigma)-\vartheta_w^{(n)}(\sigma)|^2|\langle Z_{\sigma}, \xi\rangle|^2 \rightarrow 0\quad\quad (\mbox{$n\rightarrow \infty$})
\end{equation*}
for all $\sigma\in \Gamma$.
Thus, by the dominated convergence theorem, we get
\begin{equation*}
\lim_{n\to \infty}\|\eta - \eta_n\|^2
= \lim_{n\to \infty}\sum_{\sigma \in \Gamma}|\vartheta_{u,v}(\sigma)-\vartheta_{u,v}^{(n)}(\sigma)|^2|\langle Z_{\sigma}, \xi\rangle|^2
=0,
\end{equation*}
which implies that $\eta_n$ converges to $\eta$ as $n\rightarrow \infty$,
namely the vector series
\begin{equation*}
  \sum_{j,k=0}^{\infty}w(j,k)\partial_k^*\partial_j\partial_j^*\partial_k\xi
\end{equation*}
converges naturally. Hence $\xi \in \mathrm{Dom}\, S_w$, and
\begin{equation*}
  S_w\xi
   = \sum_{j,k=0}^{\infty}w(j,k)\partial_k^*\partial_j\partial_j^*\partial_k\xi
   = \lim_{n\to \infty}\sum_{j,k=0}^nw(j,k)\partial_k^*\partial_j\partial_j^*\partial_k\xi
   = \lim_{n\to \infty}\eta_n
   = \eta,
\end{equation*}
which, together with $\eta = \sum_{\sigma \in \Gamma}\vartheta_w(\sigma)\langle Z_{\sigma},\xi\rangle Z_{\sigma}$, yields
(\ref{eq-4-7}).

Now, let $\xi \in \mathrm{Dom}\,S_w$. We show that
$\sum_{\sigma \in \Gamma}\big(\vartheta_w(\sigma)\big)^2|\langle Z_{\sigma},\xi\rangle|^2<\infty$.
In fact, by the definition of $\mathrm{Dom}\,S_w$,  the vector sequence
\begin{equation*}
   \eta_n = \sum_{j=0}^n\sum_{k=0}^nw(j,k)\partial_k^*\partial_j\partial_j^*\partial_k\xi,\quad n\geq 0,
\end{equation*}
converges in $\textsf{L}\!^2(Z)$. Thus there exists a finite constant $c\geq 0$ such that
\begin{equation*}
  \|\eta_n\|^2 \leq c,\quad \forall\, n\geq 0,
\end{equation*}
which together with (\ref{eq-4-8}) yields that
\begin{equation}\label{eq-10}
  \sum_{\sigma \in \Gamma} \big(\vartheta_w^{(n)}(\sigma)\big)^2 |\langle Z_{\sigma}, \xi\rangle|^2 = \|u_n\|^2\leq c,\quad \forall\, n\geq 0.
\end{equation}
Note that $\vartheta_w^{(n)}(\sigma)\rightarrow \vartheta_w(\sigma)$ as $n\to \infty$, for all $\sigma \in \Gamma$. Thus, by the well-known Fatou theorem,
we get from (\ref{eq-10}) that
\begin{equation*}
  \sum_{\sigma \in \Gamma} \big(\vartheta_w(\sigma)\big)^2 |\langle Z_{\sigma}, \xi\rangle|^2 \leq c< \infty.
\end{equation*}
This completes the proof.
\end{proof}

\begin{remark}\label{rem-4-1}
Let $w\colon \mathbb{N}\times \mathbb{N}\rightarrow \mathbb{R}$ be a nonnegative function satisfying condition~(\ref{eq-4-2}).
Then, by Theorem~\ref{thr-4-2}, $S_w$ has the following spectral decomposition
\begin{equation}\label{eq}
  S_w =  \sum_{\sigma \in \Gamma}\vartheta_w(\sigma)|Z_{\sigma}\rangle\!\langle Z_{\sigma}|,
\end{equation}
where $|Z_{\sigma}\rangle\!\langle Z_{\sigma}|$ is the Dirac operator associated with basis vector $Z_{\sigma}$.
In particular, $S_w$ is self-adjoint, positive and satisfies that
\begin{equation}\label{eq-eigen-vector}
S_w Z_{\sigma}=\vartheta_w(\sigma)\, Z_{\sigma},\quad \forall\, \sigma\in \Gamma,
\end{equation}
which means that each basis vector $Z_{\sigma}$ is an eigenvector of $S_w$.
\end{remark}

The next theorem gives a necessary and sufficient condition for a 2D-weighted number operator $S_w$ to be bounded.

\begin{theorem}\label{thr-4-3}
Let $w\colon \mathbb{N}\times \mathbb{N}\rightarrow \mathbb{R}$ be a nonnegative function satisfying condition~(\ref{eq-4-2}).
Then $S_w$ is bounded if and only if
\begin{equation}\label{eq-4-13}
  \sup_{\sigma\in \Gamma}\vartheta_w(\sigma)<\infty,
\end{equation}
where $\vartheta_w$ is the function on $\Gamma$ defined by (\ref{eq-4-6}). In that case, it holds true that
\begin{equation}\label{eq-4-14}
  \|S_w\|= \sup_{\sigma\in \Gamma}\vartheta_w(\sigma)
\end{equation}
and $\mathrm{Dom}\,S_w= \mathsf{L}\!^2(Z)$.
\end{theorem}

\begin{proof}
Write $b=\sup_{\sigma\in \Gamma}\vartheta_w(\sigma)$. If $b<\infty$, then by using Theorem~\ref{thr-4-2} we find
\begin{equation*}
  \|S_w\xi\|^2=\sum_{\sigma\in \Gamma}\big(\vartheta_w(\sigma)\big)^2 |\langle Z_{\sigma}, \xi\rangle|^2
  \leq b^2 \sum_{\sigma\in \Gamma} |\langle Z_{\sigma}, \xi\rangle|^2
  = b^2 \|\xi\|^2,\quad \xi\in \mathrm{Dom}\, S_w,
\end{equation*}
which implies that $S_w$ is bounded and $\|S_w\|\leq b$. Conversely, if $S_w$ is bounded, then $\|S_w\|<\infty$,
which together with Remark~\ref{rem-4-1} yields that
\begin{equation*}
  \vartheta_w(\sigma)
  =\|\vartheta_w(\sigma)Z_{\sigma}\|
  = \|S_wZ_{\sigma}\|
  \leq \|S_w\|\|Z_{\sigma}\|
  = \|S_w\|
\end{equation*}
holds for all $\sigma\in \Gamma$, which implies that $b\leq \|S_w\|$, hence $b<\infty$.
Now suppose that $b<\infty$. Then, for all $\xi\in \mathsf{L}\!^2(Z)$, we have
\begin{equation*}
  \sum_{\sigma\in \Gamma}\big(\vartheta_w(\sigma)\big)^2 |\langle Z_{\sigma}, \xi\rangle|^2
  \leq b^2\sum_{\sigma\in \Gamma} |\langle Z_{\sigma}, \xi\rangle|^2
  =b^2\|\xi\|^2
  <\infty,
\end{equation*}
which together with Theorem~\ref{thr-4-2} means that $\xi \in \mathrm{Dom}\, S_w$.
Thus $\mathrm{Dom}\,S_w= \mathsf{L}\!^2(Z)$.
Finally, combining $b\leq \|S_w\|$ with $\|S_w\|\leq b$ gives $\|S_w\|=b$.
\end{proof}

\begin{theorem}\label{thr-4-4}
Let $w\colon \mathbb{N}\times \mathbb{N}\rightarrow \mathbb{R}$ be a nonnegative function satisfying condition~(\ref{eq-4-2}),
and $\mathcal{D}=\mathrm{span}\big\{Z_{\sigma}\mid \sigma\in \Gamma\big\}$ be the linear subspace (manifold) spanned by
the canonical ONB.
Then $\mathcal{D}$ is a core for $S_w$, namely
$\mathcal{D}\subset \mathrm{Dom}\, S_w$ and $\overline{S_w|_{\mathcal{D}}} = S_w$.
\end{theorem}

\begin{proof}
By Proposition~\ref{prop-4-1}, we find that $\{Z_{\sigma}\mid \sigma\in \Gamma\} \subset \mathrm{Dom}\,S_w$, which implies
that $\mathcal{D}\subset \mathrm{Dom}\,S_w$.
Now let $\xi_0\in \mathrm{Dom}\, S_w$.
Then, by Theroem~\ref{thr-4-2}, $\sum_{\sigma\in \Gamma}\big(\vartheta_w(\sigma)\big)^2 |\langle Z_{\sigma}, \xi_0\rangle |^2<\infty$.
Set
\begin{equation*}
  \xi_n = \sum_{\sigma\in \Gamma_{n]}}\langle Z_{\sigma}, \xi_0\rangle Z_{\sigma},\quad n\geq 1,
\end{equation*}
where $\Gamma_{n]} =  \big\{\sigma\in \Gamma\mid \sigma=\emptyset\ \  \mbox{or}\ \ \max\, \sigma \leq n\big\}$.
Note that the cardinality of $\Gamma_{n]}$ is exactly $2^{n+1}$, which means that
$\xi_n$ is just a linear combination of finitely many basis vectors. Hence $\{\xi_n\mid n\geq 1\}\subset \mathcal{D}$,
in particular $\{\xi_n\mid n\geq 1\}\subset \mathrm{Dom}\, S_w$.
Clearly, $\xi_n\rightarrow \xi_0$ as $n\rightarrow \infty$.
A direct calculation gives
\begin{equation*}
\begin{split}
\|S_w\xi_0 - S_w\xi_n \|^2
  &= \sum_{\sigma\in \Gamma}\big(\vartheta_w(\sigma)\big)^2 |\langle Z_{\sigma}, \xi_0\rangle - \langle Z_{\sigma}, \xi_n\rangle|^2\\
  &= \sum_{\sigma\in \Gamma\setminus \Gamma_{n]}}\big(\vartheta_w(\sigma)\big)^2 |\langle Z_{\sigma}, \xi_0\rangle |^2\\
  &= \sum_{\sigma\in \Gamma}\big(\vartheta_w(\sigma)\big)^2 |\langle Z_{\sigma}, \xi_0\rangle |^2
     - \sum_{\sigma\in \Gamma_{n]}}\big(\vartheta_w(\sigma)\big)^2 |\langle Z_{\sigma}, \xi_0\rangle |^2,
\end{split}
\end{equation*}
which together with $\sum_{\sigma\in \Gamma}\big(\vartheta_w(\sigma)\big)^2 |\langle Z_{\sigma}, \xi_0\rangle |^2<\infty$ implies that
\begin{equation*}
  S_w|_{\mathcal{D}}\xi_n = S_w\xi_n\rightarrow S_w\xi_0.
\end{equation*}
Thus, the graph of $S_w|_{\mathcal{D}}$ is dense in that of $S_w$, namely $\mathcal{D}$ is a core of $S_w$.
\end{proof}

\begin{definition}\label{def-4-2}
Let $u\colon \mathbb{N}\rightarrow \mathbb{R}$ be a bounded nonnegative function.
The 1D-weighted number operator $N_u$ associated with $u$ is the one in $\textsf{L}\!^2(Z)$ given by
\begin{equation}\label{eq-4-15}
  N_u\xi=\sum_{k=0}^{\infty}u(k)\partial_k^*\partial_k\xi,\quad \xi \in \mathrm{Dom}\, N_u
\end{equation}
with
\begin{equation}\label{eq-4-16}
  \mathrm{Dom}\, N_u =\Big\{\, \xi \in \textsf{L}\!^2(Z) \Bigm|   \sum_{k=0}^{\infty}u(k)\partial_k^*\partial_k\xi\ \ \mbox{converges} \,\Big\}.
\end{equation}
\end{definition}

In the same way as in the proof of Theorem~\ref{thr-4-2},  we can prove the following theorem,
which characterizes the 1D-weighted number operators.

\begin{theorem}\label{thr-4-5}
Let $u\colon \mathbb{N}\rightarrow \mathbb{R}$ be a bounded nonnegative function.
Then a vector $\xi \in \mathsf{L}\!^2(Z)$ belongs to $\mathrm{Dom}\, N_u$ if and only if it satisfies that
\begin{equation}\label{eq-4-17}
 \sum_{\sigma \in \Gamma}\big(\#_u(\sigma)\big)^2|\langle Z_{\sigma},\xi\rangle|^2<\infty,
\end{equation}
where $\#_u(\sigma)=\sum_{k=0}^{\infty}\mathbf{1}_{\sigma}(k)u(k)$. In that case, one has
\begin{equation}\label{eq-4-18}
  N_u\xi = \sum_{\sigma \in \Gamma}\#_u(\sigma)\langle Z_{\sigma},\xi\rangle Z_{\sigma}.
\end{equation}
\end{theorem}

This theorem shows that a 1D-weighted number operator is also densely-defined, self-adjoint and positive.

\begin{remark}\label{rem-4-2}
A 1D-weighted number operator can be viewed as a 2D-weighted number operator.
In fact, if $N_u$ is the 1D-weighted number operator associated with a bounded nonnegative function
$u\colon \mathbb{N}\rightarrow \mathbb{R}$,
then we have
\begin{equation}\label{eq-4-19}
  N_u = S_{w^{(u)}},
\end{equation}
where $w^{(u)}$ is the function on $\mathbb{N}\times \mathbb{N}$ defined by
\begin{equation}\label{eq-4-20}
  w^{(u)}(j,k)=
  \left\{
    \begin{array}{ll}
      u(j), & \hbox{$j=k$, $(j,k)\in \mathbb{N}\times \mathbb{N}$;} \\
      0, & \hbox{$j\neq k$, $(j,k)\in \mathbb{N}\times \mathbb{N}$.}
    \end{array}
  \right.
\end{equation}
Thus, a 1D-weighted number operator admits all properties that all 2D-weighted number operators admit.
\end{remark}

\begin{example}\label{exam-4-1}
The number operator $N$ in $\mathsf{L}\!^2(Z)$  is exactly the 1D-weighted number operator $N_u$ associated with $u\equiv 1$.
\end{example}

The number operator $N$ plays an important role in many problems in mathematical physics \cite{cw,wang-chen-2016}.
The next result shows a role it plays in analyzing the 2D-weighted number operators.

\begin{theorem}\label{thr-4-6}
Let $w\colon \mathbb{N}\times \mathbb{N}\rightarrow \mathbb{R}$ be a nonnegative function satisfying condition~(\ref{eq-4-2}).
Then $\mathrm{Dom}\, N$ is a core for $S_w$.
In particular, for all bounded nonnegative function $u\colon \mathbb{N}\rightarrow \mathbb{R}$, $\mathrm{Dom}\, N$ is a core for $N_u$.
\end{theorem}

\begin{proof}
By the assumption, we have $\alpha=\sup_{k\geq 0}\sum_{j=0}^{\infty}w(j,k)<\infty$. For all $\sigma \in \Gamma$, it follows from (\ref{eq-4-6}) that
\begin{equation*}
\begin{split}
\vartheta_w(\sigma)
  &= \sum_{j=0}^{\infty}\mathbf{1}_{\sigma}(j)w(j,j) + \sum_{j,k=0}^{\infty}\big(1-\mathbf{1}_{\sigma}(j)\big)\mathbf{1}_{\sigma}(k)w(j,k)\\
  & \leq \sum_{j=0}^{\infty}\mathbf{1}_{\sigma}(j)w(j,j)   + \sum_{j,k=0}^{\infty}\mathbf{1}_{\sigma}(k)w(j,k)\\
  & = \sum_{j=0}^{\infty}\mathbf{1}_{\sigma}(j)w(j,j)   + \sum_{k=0}^{\infty}\mathbf{1}_{\sigma}(k)\sum_{j=0}^{\infty}w(j,k)\\
  & \leq \alpha\sum_{j=0}^{\infty}\mathbf{1}_{\sigma}(j)   + \alpha\sum_{k=0}^{\infty}\mathbf{1}_{\sigma}(k)\\
  &=2\alpha\#(\sigma).
\end{split}
\end{equation*}
This, together with Definition~\ref{def-3-2} and Theorem~\ref{thr-4-2}, implies that $\mathrm{Dom}\, N\subset \mathrm{Dom}\, S_w$.
On the other hand, by Definition~\ref{def-3-2}, we find that
\begin{equation*}
  \mathcal{D}=\mathrm{span}\,\{Z_{\sigma}\mid \sigma\in \Gamma\}\subset \mathrm{Dom}\, N.
\end{equation*}
Thus, by Theorem~\ref{thr-4-4} and the inclusion relation $\mathrm{Dom}\, N\subset \mathrm{Dom}\, S_w$, we know that
$\mathrm{Dom}\, N$ is a core for $S_w$.
\end{proof}

\subsection{Commutation relations}

This subsection is devoted to establish the commutation relations of the weighted number operators with the annihilation
as well as the creation operators.

\begin{theorem}\label{thr-4-7}
Let $u\colon \mathbb{N}\rightarrow \mathbb{R}$ be a bounded nonnegative function. Then,
for all $n\geq 0$, $\mathrm{Dom}\,N_u$ is an invariant subspace of $\partial_n$, and moreover
it holds on $\mathrm{Dom}\,N_u$ that
\begin{equation}\label{eq-4-21}
  N_u\partial_n = \partial_nN_u-u(n)\partial_n.
\end{equation}
\end{theorem}

\begin{proof}
Let $n\geq 0$. Then, for all $m\geq n$, by using Lemma~\ref{lem-2-2} we have
\begin{equation*}
  \Big[\sum_{k=0}^m u(k)\partial_k^*\partial_k\Big]\partial_n
  = \partial_n\sum_{k=0,k\neq n}^m u(k)\partial_k^*\partial_k
  =  \partial_n\sum_{k=0}^m u(k)\partial_k^*\partial_k - u(n)\partial_n.
\end{equation*}
If $\xi \in \mathrm{Dom}\,N_u$, then
\begin{equation*}
  \sum_{k=0}^m u(k)\partial_k^*\partial_k\big(\partial_n\xi\big) = \partial_n\sum_{k=0}^m u(k)\partial_k^*\partial_k\xi - u(n)\partial_n\xi,\quad m\geq n,
\end{equation*}
which, together with Definition~\ref{def-4-2} as well as the continuity of $\partial_n$, implies that $\partial_n\xi\in \mathrm{Dom}\,N_u$,
and taking the limit ($m\rightarrow \infty$) yields
\begin{equation*}
  N_u\partial_n\xi=N_u\big(\partial_n\xi\big) = \partial_nN_u\xi - u(n)\partial_n\xi.
\end{equation*}
This completes the proof.
\end{proof}

The above theorem establishes the commutation relations of the
1D-weighted number operators with the annihilation operators.
With the same argument, we can prove the next theorem, which establishes the commutation relations of the 1D-weighted number operators
with the creation operators.

\begin{theorem}\label{thr-4-8}
Let $u\colon \mathbb{N}\rightarrow \mathbb{R}$ be a bounded nonnegative function. Then,
for all $n\geq 0$, $\mathrm{Dom}\,N_u$ is an invariant subspace of $\partial_n^*$, and moreover
it holds on $\mathrm{Dom}\,N_u$ that
\begin{equation}\label{eq-4-22}
 N_u\partial_n^*= \partial_n^*N_u+ u(n)\partial_n^*.
\end{equation}
\end{theorem}

Recall that the number operator $N$ is exactly the 1D-weighted number operator $N_u$ associated with $u(k)\equiv 1$. Thus,
combining Theorem~\ref{thr-4-7}, Theorem~\ref{thr-4-8} with Example~\ref{exam-4-1}, we come to a useful corollary as follows.

\begin{corollary}\label{corol-4-9}
For all $j$, $k\geq 0$, $\mathrm{Dom}\, N$ is an invariant subspace of $\partial_j^*\partial_k$,
and moreover it holds on $\mathrm{Dom}\, N$ that $N\partial_j^*\partial_k = \partial_j^*\partial_kN$.
\end{corollary}

Let $w\colon \mathbb{N}\times \mathbb{N}\rightarrow \mathbb{R}$ be a nonnegative function satisfying condition~(\ref{eq-4-2}).
Then, for any fixed $n\geq 0$, it can be shown that the function $j\mapsto w(j,n)$ is a bounded nonnegative function on $\mathbb{N}$,
thus by Definition~\ref{def-4-2} we have the 1D-weighted number operator $N_{w(\cdot,n)}$ associated with function $w(\cdot,n)$.
Similarly, we have the 1D-weighted number operator $N_{w(n,\cdot)}$ associated with function $w(n,\cdot)$.

\begin{theorem}\label{thr-CR-A}
Let $w\colon \mathbb{N}\times \mathbb{N}\rightarrow \mathbb{R}$ be a nonnegative function satisfying condition~(\ref{eq-4-2}). Then,
for all $n\geq 0$, it holds on $\mathrm{Dom}\, N$ that
\begin{equation}\label{eq-CR-A}
  S_w\partial_n
      = \partial_nS_w
        + \partial_nN_{w(\cdot,n)}
    +\partial_nN_{w(n,\cdot)} -\Big[ 2w(n,n) +\sum_{j=0}^{\infty}w(j,n)\Big]\partial_n,
\end{equation}
where $N_{w(\cdot,n)}$, $N_{w(n,\cdot)}$ are the 1D-weighted number operators associated with functions $w(\cdot,n)$ and $w(n,\cdot)$, respectively.
\end{theorem}

\begin{proof}
Let $n\geq 0$. Then, for $m\geq n$, we have
\begin{equation}\label{eq}
\begin{split}
  \sum_{j,k=0}^m w(j,k)\partial_k^*\partial_j\partial_j^*\partial_k
  = & \sum_{j=0,j\neq n}^m \sum_{k=0,k\neq n}^m w(j,k)\partial_k^*\partial_j\partial_j^*\partial_k
      + \partial_n^*\partial_n\Big(\sum_{j=0,j\neq n}^mw(j,n)\partial_j\partial_j^*\Big)\\
   & + \partial_n\partial_n^*\Big(\sum_{k=0,k\neq n}^mw(n,k)\partial_k^*\partial_k\Big) + w(n,n) \partial_n^*\partial_n.\\
\end{split}
\end{equation}
Using this formula and Lemma~\ref{lem-3-2}, we can obtain the following relation for $m\geq n$
\begin{equation*}
\begin{split}
  \Big(\sum_{j,k=0}^m & w(j,k)\partial_k^*\partial_j\partial_j^*\partial_k\Big)\partial_n\\
  &=  \Big(\sum_{j=0,j\neq n}^m \sum_{k=0,k\neq n}^m w(j,k)\partial_k^*\partial_j\partial_j^*\partial_k\Big)\partial_n
      + \partial_n^*\partial_n\Big(\sum_{j=0,j\neq n}^mw(j,n)\partial_j\partial_j^*\Big)\partial_n\\
   &\quad + \partial_n\partial_n^*\Big(\sum_{k=0,k\neq n}^mw(n,k)\partial_k^*\partial_k\Big)\partial_n + w(n,n) \partial_n^*\partial_n\partial_n\\
  &= \partial_n \Big(\sum_{j=0,j\neq n}^m \sum_{k=0,k\neq n}^m w(j,k)\partial_k^*\partial_j\partial_j^*\partial_k\Big)
      + \partial_n^*\partial_n\partial_n\Big(\sum_{j=0,j\neq n}^mw(j,n)\partial_j\partial_j^*\Big)\\
   &\quad + \partial_n\partial_n^*\partial_n\Big(\sum_{k=0,k\neq n}^mw(n,k)\partial_k^*\partial_k\Big) + w(n,n) \partial_n^*\partial_n\partial_n\\
  &= \partial_n \Big(\sum_{j=0,j\neq n}^m \sum_{k=0,k\neq n}^m w(j,k)\partial_k^*\partial_j\partial_j^*\partial_k\Big)
   +  \partial_n\Big(\sum_{k=0,k\neq n}^mw(n,k)\partial_k^*\partial_k\Big),
\end{split}
\end{equation*}
which together with
\begin{equation*}
\begin{split}
\partial_n & \Big(\sum_{j=0,j\neq n}^m \sum_{k=0,k\neq n}^m w(j,k)\partial_k^*\partial_j\partial_j^*\partial_k\Big)\\
&=  \partial_n\sum_{j,k=0}^m  w(j,k)\partial_k^*\partial_j\partial_j^*\partial_k
   - w(n,n)\partial_n
   -\Big(\sum_{j=0}^mw(j,n)\Big)\partial_n + \partial_n\sum_{j=0}^mw(j,n)\partial_j^*\partial_j\\
\end{split}
\end{equation*}
and
\begin{equation*}
  \partial_n\Big(\sum_{k=0,k\neq n}^mw(n,k)\partial_k^*\partial_k\Big)
  = \partial_n\sum_{k=0}^mw(n,k)\partial_k^*\partial_k - w(n,n)\partial_n
\end{equation*}
gives
\begin{equation}\label{eq-4-23}
\begin{split}
  \Big(&\sum_{j,k=0}^m  w(j,k)\partial_k^*\partial_j\partial_j^*\partial_k\Big)\partial_n\\
     & = \partial_n\sum_{j,k=0}^m  w(j,k)\partial_k^*\partial_j\partial_j^*\partial_k
       + \partial_n\sum_{j=0}^mw(j,n)\partial_j^*\partial_j
   +\partial_n\sum_{k=0}^mw(n,k)\partial_k^*\partial_k\\
   &\quad - 2\Big[w(n,n) +\sum_{j=0}^mw(j,n)\Big]\partial_n.
\end{split}
\end{equation}
Now let $\xi \in \mathrm{Dom}\, N$. Then, by Theorem~\ref{thr-4-6}, we know that
\begin{equation*}
\xi \in \mathrm{Dom}\, S_w \cap \mathrm{Dom}\, N_{w(\cdot,n)}\cap \mathrm{Dom}\, N_{w(n,\cdot)}.
\end{equation*}
Hence, by (\ref{eq-4-23}), we have
\begin{equation*}
\begin{split}
  \Big(&\sum_{j,k=0}^m  w(j,k)\partial_k^*\partial_j\partial_j^*\partial_k\Big)\partial_n\xi\\
     & = \partial_n\sum_{j,k=0}^m  w(j,k)\partial_k^*\partial_j\partial_j^*\partial_k\xi
       + \partial_n\Big(\sum_{j=0}^mw(j,n)\partial_j^*\partial_j\xi\Big)
   +\partial_n\Big(\sum_{k=0}^mw(n,k)\partial_k^*\partial_k\xi\Big)\\
   &\quad - \Big[2w(n,n) +\sum_{j=0}^mw(j,n)\Big]\partial_n\xi.
\end{split}
\end{equation*}
which, together with Definition~\ref{def-4-1}, Definition~\ref{def-4-2} and the limit procedure ($m\rightarrow \infty$),
yields
\begin{equation*}
 S_w\partial_n\xi
      = \partial_nS_w\xi
        + \partial_nN_{w(\cdot,n)}\xi
    +\partial_nN_{w(n,\cdot)}\xi - \Big[2w(n,n) +\sum_{j=0}^{\infty}w(j,n)\Big]\partial_n\xi.
\end{equation*}
Here we make use of the continuity of $\partial_n$.
\end{proof}

Theorem~\ref{thr-CR-A} establishes the commutation relations of the 2D-weighted number operators with the annihilation operators.
With the same argument, we can also get the commutation relations
of the 2D-weighted number operators with the creation operators as follows.

\begin{theorem}\label{thr-CR-C}
Let $w\colon \mathbb{N}\times \mathbb{N}\rightarrow \mathbb{R}$ be a nonnegative function satisfying condition~(\ref{eq-4-2}). Then,
for all $n\geq 0$, it holds on $\mathrm{Dom}\, N$ that
\begin{equation}\label{eq-CR-C}
  S_w\partial_n^* = \partial_n^*S_w  - \partial_n^* N_{w(\cdot,n)}- \partial_n^*N_{w(n,\cdot)}
   + \Big(\sum_{j=0}^{\infty}w(j,n)\Big)\partial_n^*,
\end{equation}
where $N_{w(\cdot,n)}$, $N_{w(n,\cdot)}$ are the 1D-weighted number operators associated with functions $w(\cdot,n)$ and $w(n,\cdot)$, respectively.
\end{theorem}

As is seen, the 2D-weighted number operators do not commutate with the annihilation and creation operators
in general. However, the case becomes much better when they meet with the product operator $\partial_n^*\partial_n$.

\begin{theorem}\label{thr-4-12}
Let $w\colon \mathbb{N}\times \mathbb{N}\rightarrow \mathbb{R}$ be a nonnegative function satisfying condition~(\ref{eq-4-2}). Then,
for all $n\geq 0$, $\mathrm{Dom}\,S_w$ is an invariant subspace of $\partial_n^*\partial_n$, and moreover
it holds on $\mathrm{Dom}\,S_w$ that
\begin{equation}\label{eq-4-27}
  S_w\partial_n^*\partial_n = \partial_n^*\partial_nS_w.
\end{equation}
\end{theorem}

\begin{proof}
We use Theorem~\ref{thr-4-2}. Let $n\geq 0$ and $\xi\in \mathrm{Dom}\,S_w$.
Then, by Lemma~\ref{lem-3-3} we have
\begin{equation*}
  \sum_{\sigma \in \Gamma}\big(\vartheta_w(\sigma)\big)^2|\langle Z_{\sigma},\partial_n^*\partial_n\xi\rangle|^2
   = \sum_{\sigma \in \Gamma}\big(\vartheta_w(\sigma)\big)^2\mathbf{1}_{\sigma}(n)|\langle Z_{\sigma},\xi\rangle|^2
   \leq \sum_{\sigma \in \Gamma}\big(\vartheta_w(\sigma)\big)^2|\langle Z_{\sigma},\xi\rangle|^2,
\end{equation*}
which, together with Theorem~\ref{thr-4-2} as well as the self-adjointness of operator $\partial_n^*\partial_n$,
implies that $\partial_n^*\partial_n\xi \in \mathrm{Dom}\,S_w$ and
\begin{equation*}
  S_w\partial_n^*\partial_n\xi
  = \sum_{\sigma\in \Gamma}\vartheta_w(\sigma)\langle Z_{\sigma}, \partial_n^*\partial_n\xi\rangle Z_{\sigma}
   = \sum_{\sigma\in \Gamma}\vartheta_w(\sigma)\mathbf{1}_{\sigma}(n)\langle Z_{\sigma}, \xi\rangle Z_{\sigma}.
\end{equation*}
On the other hand, it follows from the continuity of operator $\partial_n^*\partial_n$ that
\begin{equation*}
  \partial_n^*\partial_nS_w\xi
   = \sum_{\sigma\in \Gamma}\vartheta_w(\sigma)\langle Z_{\sigma}, \xi\rangle \partial_n^*\partial_nZ_{\sigma}
   = \sum_{\sigma\in \Gamma}\vartheta_w(\sigma)\mathbf{1}_{\sigma}(n)\langle Z_{\sigma}, \xi\rangle Z_{\sigma}.
\end{equation*}
Thus $ S_w\partial_n^*\partial_n\xi = \partial_n^*\partial_nS_w\xi$.
\end{proof}

\section{Quantum exclusion semigroup}\label{sec-5}

In the final section, we address the problem to construct a quantum Markov semigroup (QMS, for short) with a formal generator of form (\ref{eq-1-2}).
From a physical point of view, such a QMS belongs to the category of quantum exclusion semigroups, and might serve as a model describing
an open quantum system consisting of an arbitrary number of identical Fermi particles,
where operator $\partial_j^*\partial_k$ represents the jump of particles from site $k$ to site $j$ with
the jump rate $\sqrt{w(j,k)}$.

In what follows, we denote by $\mathcal{D}$ the linear subspace (manifold) spanned by the ONB of $\mathsf{L}\!^2(Z)$, namely
$\mathcal{D} = \mathrm{span}\, \big\{Z_{\sigma} \mid \sigma\in \Gamma\big\}$.
For $n\geq 0$, we set
\begin{equation}\label{eq-5-1}
  \Gamma_{n]} = \big\{ \sigma\in \Gamma \mid \sigma =\emptyset\ \ \mbox{or}\ \max \sigma\leq n\big\},
\end{equation}
where $\max \sigma$ means the biggest of  integers contained in $\sigma$ for nonempty $\sigma\in \Gamma$.

\begin{definition}\label{def-5-1}
A nonnegative function $w\colon \mathbb{N}\times \mathbb{N}\rightarrow \mathbb{R}$ is called a transition kernel if it
satisfies that
\begin{equation}\label{eq-5-2}
\sup_{k\geq 0}\sum_{j=0}^{\infty}w(j,k)<\infty.
\end{equation}
A transition kernel $w$ is said to be regular if it holds that $\inf_{j\geq 0}w(j,j)>0$.
\end{definition}

We mention that there do exist regular transition kernels on $\mathbb{N}\times \mathbb{N}$. In fact, we have a regular transition kernel $w_0$
given by
\begin{equation*}
  w_0(j,k)=
  \left\{
    \begin{array}{ll}
      1, & \hbox{$j= k$, $(j,k) \in \mathbb{N}\times \mathbb{N}$;} \\
      0, & \hbox{$j\neq k$, $(j,k) \in \mathbb{N}\times \mathbb{N}$.}
    \end{array}
  \right.
\end{equation*}
This transition kernel is called the canonical transition kernel on $\mathbb{N}\times \mathbb{N}$.

According to Section~\ref{sec-4}, for a given transition kernel $w$,
the 2D-weighted number operator $S_w$ associated with $w$ is a densely defined, positive and self-adjoint operator in $\mathsf{L}\!^2(Z)$, and moreover
it admits the inclusion relation $\mathrm{Dom}\,N\subset \mathrm{Dom}\,S_w$.

\begin{definition}\label{def-5-2}
Let  $w$ be a transition kernel.
An operator $G$ densely defined in $\mathsf{L}\!^2(Z)$ is said to be $S_w$-admissible if both $G$ and its adjoint $G^*$ have
the same domain as $S_w$ and further satisfy that
\begin{equation}\label{eq-5-3}
 G+G^* =- S_w
\end{equation}
on $\mathrm{Dom}\,S_w$.
\end{definition}

\begin{theorem}\label{thr-5-1}
Let $w$ be a transition kernel and
$G$ be an operator densely defined in $\mathsf{L}\!^2(Z)$. Suppose that
$G$ satisfies the following two requirements:\ \
\begin{enumerate}
  \item[(1)] $G$ is the infinitesimal generator of a strongly continuous contraction semigroup
$P=(P_t)_{t \geq 0}$  on $\mathsf{L}^2(Z)$;
  \item[(2)] and $G$ is $S_w$-admissible.
\end{enumerate}
Then there exists a minimal QDS $\mathcal{T} =(\mathcal{T}_t)_{t\geq 0}$ on $\mathfrak{B}(\mathsf{L}^2(Z))$
satisfying that
\begin{equation}\label{eq-5-4}
\begin{split}
 &\big\langle \xi,\mathcal{T}_t(X) \eta\big\rangle\\
 & =\big\langle \xi,X \eta\big\rangle
  +\int_0^t\Big[\big\langle \xi,\mathcal{T}_s(X) G\eta\big\rangle + \big\langle G\xi, \mathcal{T}_s(X)\eta\big\rangle
  + \sum_{j,k=0}^{\infty}w(j,k)\big\langle \partial_j^*\partial_k \xi, \mathcal{T}_s(X) \partial_j^*\partial_k \eta\big\rangle\Big]ds
 \end{split}
\end{equation}
for all $\xi$, $\eta \in \mathrm{Dom}\, G$ and all $X\in \mathfrak{B}(\mathsf{L}^2(Z))$.
\end{theorem}

\begin{proof}
Set $L_{jk}=\sqrt{w(j,k)}\,\partial_j^*\partial_k$ for $j$, $k\geq 0$. Then,
for all $j$, $k\geq 0$, the domain of $L_{jk}$ clearly contains the domain of $G$ since $\partial_j^*\partial_k$ is a bounded operator on
whole $\mathsf{L}\!^2(Z)$.
For all $\xi$, $\eta\in \mathrm{Dom}\,G$, it follows from the assumption that $\xi$, $\eta\in \mathrm{Dom}\,S_w$,
hence by Theorem~\ref{thr-4-2}
\begin{equation*}
  \sum_{j,k=0}^{\infty}L_{jk}^*L_{jk}\eta
  =\sum_{j,k=0}^{\infty}w(j,k)\partial_k\partial_j\partial_j^*\partial_k \eta,
  = S_w\eta
\end{equation*}
which, together with the assumption of $G$ being $S_w$-admissible, implies that
\begin{equation*}
  \langle \xi, G\eta\rangle + \langle G\xi,\eta\rangle + \sum_{j,k=0}^{\infty}\langle L_{jk}\xi, L_{jk} \eta\rangle
  =\langle \xi, (G+G^*)\eta\rangle + \langle \xi, S_w \eta\rangle
  =0.
\end{equation*}
This shows that operators $G$ and $(L_{jk})_{j,k\geq 0}$ satisfy all the conditions in Lemma~\ref{lem-2-1},
thus there exists a minimal QDS $\mathcal{T} =(\mathcal{T}_t)_{t\geq 0}$ on $\mathfrak{B}(\mathsf{L}^2(Z))$
satisfying that
\begin{equation*}
\begin{split}
 &\big\langle \xi,\mathcal{T}_t(X) \eta\big\rangle\\
 & =\big\langle \xi,X \eta\big\rangle
  +\int_0^t\Big[\big\langle \xi,\mathcal{T}_s(X) G\eta\big\rangle + \big\langle G\xi, \mathcal{T}_s(X)\eta\big\rangle
  + \sum_{j,k=0}^{\infty}\big\langle L_{jk} \xi, \mathcal{T}_s(X) L_{jk}\eta\big\rangle\Big]ds
 \end{split}
\end{equation*}
for all $\xi$, $\eta \in \mathrm{Dom}\, G$ and all $X\in \mathfrak{B}(\mathsf{L}^2(Z))$, which is exactly the same as (\ref{eq-5-4}).
\end{proof}

\begin{remark}\label{rem-5-1}
We call the QDS $\mathcal{T} =(\mathcal{T}_t)_{t\geq 0}$ mentioned in Theorem~\ref{thr-5-1} the
minimal QDS constructed from $w$, $G$ and $\big(\partial_j^*\partial_k\big)_{j,k\geq 0}$.
By letting $H=\mathrm{i}\big(G+\frac{1}{2}S_w\big)$, one immediately gets the formal generator of $\mathcal{T} =(\mathcal{T}_t)_{t\geq 0}$,
which is exactly the same as (\ref{eq-1-2}).
\end{remark}

Next, we consider the conservativity of the minimal QDS constructed from $w$, $G$ and $\big(\partial_j^*\partial_k\big)_{j,k\geq 0}$.
To do so, we first make some necessary preparations.

\begin{definition}\label{def-5-3}
An operator $A$ in $\mathsf{L}\!^2(Z)$ with $\mathcal{D}\subset \mathrm{Dom}\, A$
is said to be amenable if for all $\xi\in \mathrm{Dom}\, A$ the sequence $\big(A\xi_n\big)_{n\geq 1}$ converges in  $\mathsf{L}\!^2(Z)$,
where
\begin{equation}\label{eq-5-5}
  \xi_n = \sum_{\sigma\in \Gamma_{n]}}\langle Z_{\sigma}, \xi\rangle  Z_{\sigma}
\end{equation}
with $\Gamma_{n]}$ being the one defined by (\ref{eq-5-1}).
\end{definition}

\begin{proposition}\label{prop-5-2}
Let  $w$ be a transition kernel. Then $S_w$ is amenable. In particular,
the number operator $N$ is amenable.
\end{proposition}

\begin{proof}
In fact, by Theorem~\ref{thr-4-4} we know that $\mathcal{D}\subset \mathrm{Dom}\, S_w$. Let $\xi \in \mathrm{Dom}\, S_w$.
Then, by Theorem~\ref{thr-4-2}, we have
\begin{equation*}
  \sum_{\sigma \in \Gamma}\big(\vartheta_w(\sigma)\big)^2|\langle Z_{\sigma},\xi\rangle|^2<\infty,
\end{equation*}
which together with
\begin{equation*}
\begin{split}
\big\|S_w\xi-S_w\xi_n\big\|^2
  &= \sum_{\sigma \in \Gamma}\big(\vartheta_w(\sigma)\big)^2|\langle Z_{\sigma},\xi-\xi_n\rangle|^2\\
  & = \sum_{\sigma \in \Gamma}\big(\vartheta_w(\sigma)\big)^2|\langle Z_{\sigma},\xi\rangle|^2
    - \sum_{\sigma \in \Gamma_{n]}}\big(\vartheta_w(\sigma)\big)^2|\langle Z_{\sigma},\xi\rangle|^2
\end{split}
\end{equation*}
implies that $\|S_w\xi-S_w\xi_n\big\|\rightarrow 0$, hence $\big(S_w\xi_n\big)_{n\geq 1}$ converges in  $\mathsf{L}\!^2(Z)$.
\end{proof}

\begin{proposition}\label{prop-5-3}
Let $w$ be a regular transition kernel. Then the following statements hold true:
\begin{enumerate}
  \item[(1)] $\mathrm{Dom}\,S_w= \mathrm{Dom}\,N$;
  \item[(2)] $\big\langle \xi, S_w \xi\big\rangle \leq 2\alpha \big\langle \xi, N \xi\big\rangle$, $\forall\, \xi\in \mathrm{Dom}\,S_w$,
  where $\alpha=\sup_{k\geq 0}\sum_{j=0}^{\infty}w(j,k)$;
  \item[(3)] and\ $\big\langle S_w\xi, N \xi\big\rangle \leq \frac{1}{\beta} \big\langle S_w\xi, S_w \xi\big\rangle$, $\forall\, \xi\in \mathrm{Dom}\,S_w$,
  where $\beta=\sup_{j\geq 0}w(j,j)$.
\end{enumerate}
\end{proposition}

\begin{proof}
(1)\  According to Theorem~\ref{thr-4-6}, we need only to verify $\mathrm{Dom}\,S_w\subset \mathrm{Dom}\,N$.
Let $\xi \in \mathrm{Dom}\,S_w$. Then, by Theorem~\ref{thr-4-2}, we have
\begin{equation*}
  \sum_{\sigma \in \Gamma}\big(\vartheta_w(\sigma)\big)^2|\langle Z_{\sigma},\xi\rangle|^2<\infty.
\end{equation*}
On the other hand, it follows from the regularity of $w$ that
\begin{equation*}
  \#(\sigma)=\sum_{j=0}^{\infty}\mathbf{1}_{\sigma}(j)
  \leq \frac{1}{\beta}\sum_{j=0}^{\infty}\mathbf{1}_{\sigma}(j)w(j,j)
  \leq  \frac{1}{\beta}\vartheta_w(\sigma),\quad \forall\, \sigma\in \Gamma.
\end{equation*}
Thus
\begin{equation*}
  \sum_{\sigma \in \Gamma}\big(\#(\sigma)\big)^2|\langle Z_{\sigma},\xi\rangle|^2
    =\frac{1}{\beta^2} \sum_{\sigma \in \Gamma}\big(\vartheta_w(\sigma)\big)^2|\langle Z_{\sigma},\xi\rangle|^2<\infty,
\end{equation*}
which together with Definition~\ref{def-3-2} means that $\xi\in \mathrm{Dom}\,N$. It then follows from the arbitrariness of the choice $\xi \in \mathrm{Dom}\,S_w$ that $\mathrm{Dom}\,S_w\subset \mathrm{Dom}\,N$.

(2)\ Let $\xi\in \mathrm{Dom}\,S_w$. By Theorem~\ref{thr-4-2} as well as the expansion of $\xi$ with respect to the canonical ONB, we find
\begin{equation*}
  \big\langle \xi, S_w \xi\big\rangle
  = \sum_{\sigma\in \Gamma} \overline{\langle Z_{\sigma}, \xi\rangle} \vartheta_w(\sigma)\langle Z_{\sigma},\xi\rangle
  = \sum_{\sigma\in \Gamma} \vartheta_w(\sigma)|\langle Z_{\sigma},\xi\rangle|^2.
\end{equation*}
On the other hand, for each $\sigma\in \Gamma$, in view of the equality $\sup_{j\geq 0}w(j,j)\leq \alpha$, we have
\begin{equation*}
\begin{split}
  \vartheta_w(\sigma)
     & =\sum_{j=0}^{\infty}\mathbf{1}_{\sigma}(j)w(j,j)
        +\sum_{j,k=0}^{\infty}\big( 1- \mathbf{1}_{\sigma}(j)\big)\mathbf{1}_{\sigma}(k)w(j,k)\\
     & =\sum_{j=0}^{\infty}\mathbf{1}_{\sigma}(j)w(j,j)
        +\sum_{k=0}^{\infty} \mathbf{1}_{\sigma}(k) \Big[\sum_{j=0}^{\infty} \big( 1- \mathbf{1}_{\sigma}(j)\big) w(j,k)\Big]\\
     & \leq \alpha\sum_{j=0}^{\infty}\mathbf{1}_{\sigma}(j)
        +\sum_{k=0}^{\infty} \mathbf{1}_{\sigma}(k) \Big[\sum_{j=0}^{\infty}  w(j,k)\Big]\\
     & \leq \alpha\sum_{j=0}^{\infty}\mathbf{1}_{\sigma}(j)
        +\alpha\sum_{k=0}^{\infty} \mathbf{1}_{\sigma}(k) \\
     & = 2\alpha\#(\sigma).
\end{split}
\end{equation*}
Thus
\begin{equation*}
  \big\langle \xi, S_w \xi\big\rangle
  = \sum_{\sigma\in \Gamma} \vartheta_w(\sigma)|\langle Z_{\sigma},\xi\rangle|^2
  \leq 2\alpha\sum_{\sigma\in \Gamma} \#(\sigma)|\langle Z_{\sigma},\xi\rangle|^2
  = 2\alpha \big\langle \xi, N \xi\big\rangle.
\end{equation*}

(3)\ Let $\xi\in \mathrm{Dom}\,S_w$. Then, with the same argument as in the proof of (2), we can get
\begin{equation*}
  \big\langle S_w\xi, N \xi\big\rangle
  = \sum_{\sigma\in \Gamma} \vartheta_w(\sigma)\#(\sigma)|\langle Z_{\sigma},\xi\rangle|^2
  \leq \frac{1}{\beta}\sum_{\sigma\in \Gamma} \big(\vartheta_w(\sigma)\big)^2|\langle Z_{\sigma},\xi\rangle|^2
  = \frac{1}{\beta}\big\langle S_w\xi, S_w \xi\big\rangle.
\end{equation*}
Here we make use of the inequality $\#(\sigma)\leq \frac{1}{\beta}\vartheta_w(\sigma)$, which is proven above (see the proof of (1) for details).
\end{proof}

\begin{theorem}\label{thr-5-4}
Let $w$ be a regular transition kernel and $G$ be a densely-defined operator in $\mathsf{L}\!^2(Z)$. Suppose that
$G$ satisfies the following three requirements:\ \
\begin{enumerate}
  \item[(1)] $G$ is the infinitesimal generator of a strongly continuous contraction semigroup $P=(P_t)_{t \geq 0}$  on $\mathsf{L}^2(Z)$;
  \item[(2)] $G$ is $S_w$-admissible and amenable;
  \item[(3)] and there exists a constant $b$ such that, for all $\xi\in \mathcal{D}$, the following inequality holds
  \begin{equation}\label{eq-5-6}
    2\Re\big\langle N\xi, G\xi \big\rangle + \sum_{j,k=0}^{\infty}w(j,k) \big\langle \partial_j^*\partial_k\xi, N \partial_j^*\partial_k\xi \big\rangle
    \leq b  \big\langle \xi, N\xi \big\rangle,
  \end{equation}
  where $\mathcal{D}=\mathrm{span}\,\{Z_{\sigma}\mid \sigma\in \Gamma\}$ is the linear subspace (manifold) spanned by the canonical ONB
  $\{Z_{\sigma}\mid \sigma\in \Gamma\}$.
\end{enumerate}
Then the minimal QDS $\mathcal{T} =(\mathcal{T}_t)_{t\geq 0}$ constructed from $w$, $G$ and $(\partial_j^*\partial_k)_{j,k\geq 0}$ is conservative.
\end{theorem}

\begin{proof}
We first note that all the conditions of Theorem~\ref{thr-5-1} are satisfied.
Thus it makes sense to say the minimal QDS constructed from $w$, $G$ and $(\partial_j^*\partial_k)_{j,k\geq 0}$.

Now let $\mathcal{T} =(\mathcal{T}_t)_{t\geq 0}$ be the minimal QDS constructed from $w$, $G$ and $(\partial_j^*\partial_k)_{j,k\geq 0}$.
We show that $\mathcal{T} =(\mathcal{T}_t)_{t\geq 0}$ is conservative.
To do so, we set
\begin{equation*}
L_{jk}=\sqrt{w(j,k)}\,\partial_j^*\partial_k
\end{equation*}
for $j$, $k\geq 0$ and take $C=2\alpha N$, $\Phi=S_w$.
Then, $\mathcal{T} =(\mathcal{T}_t)_{t\geq 0}$ is also the minimal QDS constructed from $G$ and $(L_{jk})_{j,k\geq 0}$.

Clearly, $C$ is a self-adjoint operator in $\mathsf{L}\!^2(Z)$ and,
by Theorem~\ref{thr-4-4}, Remark~\ref{rem-4-2} as well as Example~\ref{exam-4-1}, $\mathcal{D}$ is a core for $C$.
By Proposition~\ref{prop-5-3}, Definition~\ref{def-5-2} as well as requirement~(2) of the present theorem, we have
\begin{equation*}
  \mathrm{Dom}\, C = \mathrm{Dom}\,N = \mathrm{Dom}\,S_w = \mathrm{Dom}\,G.
\end{equation*}
By Definition~\ref{def-5-3}, Proposition~\ref{prop-5-2} and requirement~(2), we know that for all $\xi\in \mathrm{Dom}\, C$ there exists
a sequence $\big(\xi_n\big)_{n\geq 1}\subset \mathcal{D}$ such that both $\big(G\xi_n\big)_{n\geq 1}$ and $\big(C\xi_n\big)_{n\geq 1}$ converge in $\mathsf{L}\!^2(Z)$.
By Theorem~\ref{thr-4-2} and Remark~\ref{rem-4-1}, $\Phi$ is a positive self-adjoint operator.
For all $\xi\in \mathcal{D}$, by Definition~\ref{def-5-2} and requirement~(2), we find
\begin{equation*}
  -2\Re\langle \xi,G\xi\rangle
  = -[\langle \xi,G\xi\rangle + \langle G\xi,\xi\rangle]
   = -\langle \xi,(G + G^*)\xi\rangle
   = \langle \xi,\Phi\xi\rangle
   \leq \langle \xi,C\xi\rangle.
\end{equation*}
It follows easily from Lemma~\ref{lem-3-1} that
$L_{jk}(\mathcal{D}) \subset \mathcal{D} \subset \mathrm{Dom}\, C$ holds for all $j$, $k\geq 0$.
Finally, by requirement~(3), we have
\begin{equation*}
2\Re\big\langle C\xi, G\xi \big\rangle + \sum_{j,k=0}^{\infty} \big\langle L_{jk}\xi, N L_{jk}\xi \big\rangle
    \leq b  \big\langle \xi, C\xi \big\rangle,\quad \forall\, \xi\in \mathcal{D}.
\end{equation*}
Therefore, operators $G$ and $(L_{jk})_{j,k\geq 0}$ satisfy all the conditions in Lemma~\ref{lem-2-2},
which means that $\mathcal{T} =(\mathcal{T}_t)_{t\geq 0}$ is conservative.
\end{proof}

As is seen, Theorem~\ref{thr-5-4} actually offers a sufficient condition for the existence of
a QMS with a formal generator of form (\ref{eq-1-2}).
The next example then shows the effectiveness of this condition.

\begin{example}\label{exam-5-1}
Let $f\colon \Gamma \rightarrow \mathbb{R}$ be a bounded function and define $H_f$ as
\begin{equation}\label{eq-5-7}
  H_f= \sum_{\sigma\in \Gamma}f(\sigma)|Z_{\sigma}\rangle\!\langle Z_{\sigma}|.
\end{equation}
Then $H_f$ is a bounded self-adjoint operator on $\mathsf{L}\!^2(Z)$ with $\mathrm{Dom}\, H_f=\mathsf{L}\!^2(Z)$.
Let $w$ be a regular transition kernel and
$G= -\mathrm{i}H_f-\frac{1}{2}S_w$. Then, it is not hard to verify that $w$ and $G$ satisfy all the requirements in Theorem~\ref{thr-5-4}.
Thus there exists a minimal QMS $\mathcal{T} =(\mathcal{T}_t)_{t\geq 0}$ on $\mathfrak{B}(\mathsf{L}^2(Z))$ satisfying
(\ref{eq-5-4}).
\end{example}

\begin{remark}\label{rem-5-2}
From Theorem~\ref{thr-5-1}, Theorem~\ref{thr-5-4} and Example~\ref{exam-5-1}, one naturally comes to the conclusion that under some mild conditions
there does exist a QMS with a formal generator of form (\ref{eq-1-2}).
Such a QMS belongs to the category of quantum exclusion semigroups.
\end{remark}

We now show some basic properties of the QMS mentioned in Example~\ref{exam-5-1}.
Recall that $\mathbb{E}$ denotes the expectation with respect to $\mathbb{P}$ (see Section~\ref{sec-3} for details),
which is a projection operator on $\mathsf{L}\!^2(Z)$, thus belongs to $\mathfrak{B}(\mathsf{L}\!^2(Z))$.

\begin{proposition}\label{prop-5-5}
Let $f\colon \Gamma \rightarrow \mathbb{R}$ be a bounded function and $w$ a regular transition kernel.
Let $\mathcal{T} =(\mathcal{T}_t)_{t\geq 0}$ be the minimal QMS on $\mathfrak{B}(\mathsf{L}^2(Z))$ satisfying (\ref{eq-5-4}),
where $G= -\mathrm{i}H_f-\frac{1}{2}S_w$ with $H_f$ being given by (\ref{eq-5-7}). Then, the expectation $\mathbb{E}$
is a subharmonic projection operator for $\mathcal{T}$, namely it satisfies that\
$\mathbb{E} \leq \mathcal{T}_t(\mathbb{E})$ for all $t\geq 0$.
\end{proposition}

\begin{proof}
It can be verified that the contraction semigroup $(P_t)_{t\geq 0}$ generated by $G=-\mathrm{i}H_f-\frac{1}{2}S_w$ can be represented as
\begin{equation}\label{eq-5-8}
  P_t\xi
   = \sum_{\sigma \in \Gamma} e^{-[\mathrm{i}f(\sigma)+\vartheta_w(\sigma)/2]t} \langle Z_{\sigma},\xi\rangle Z_{\sigma},\quad t\geq 0,\, \xi \in \mathsf{L}\!^2(Z).
\end{equation}
On the other hand, we can easily find that the range of $\mathbb{E}$ has the following form
\begin{equation*}
\mathrm{Ran}\,\mathbb{E} = \{\alpha Z_{\emptyset} \mid \alpha \in \mathbb{C}\},
\end{equation*}
which together with (\ref{eq-5-8})
shows that $\mathrm{Ran}\,\mathbb{E}$ is invariant under the operators $P_t$, $t\geq 0$.
Now let $\xi\in \mathrm{Dom}\,G \cap \mathrm{Ran}\,\mathbb{E}$ and $j$, $k\geq 0$.
Then $\xi=\lambda Z_{\emptyset}$ for some $\lambda \in \mathbb{C}$. Clearly
$\partial_j^*\partial_k\xi =0$, which gives
\begin{equation*}
  \partial_j^*\partial_k\xi=\mathbb{E}\partial_j^*\partial_k\xi.
\end{equation*}
Thus, by Theorem III.1 in \cite{fag-reb}, $\mathbb{E}$ is a subharmonic projection for $\mathcal{T}$.
\end{proof}

Decoherence-free subalgebras play a fundamental role in the the study of decoherence.
As in \cite{dhahri}, we define the decoherence-free subalgebra $\mathfrak{N}(\mathcal{T})$ of $\mathcal{T}$ as follows.
\begin{equation}\label{eq-5-9}
\begin{split}
&\mathfrak{N}(\mathcal{T})\\
&=\big\{ X \in \mathfrak{B}(\mathsf{L}\!^2(Z)) \mid
 \mathcal{T}_t(X^*X)=\mathcal{T}_t(X^*)\mathcal{T}_t(X),\
      \mathcal{T}_t(XX^*)=\mathcal{T}_t(X)\mathcal{T}_t(X^*),\ \ t\geq 0\big\}.
\end{split}
\end{equation}
According to \cite{dhahri}, $\mathfrak{N}(\mathcal{T})$ is a von Neumann subalgebra of $\mathfrak{B}(\mathsf{L}\!^2(Z))$.

\begin{proposition}\label{prop-5-6}
Let $f\colon \Gamma \rightarrow \mathbb{R}$ be a bounded function and $w$ a regular transition kernel.
Let $\mathcal{T} =(\mathcal{T}_t)_{t\geq 0}$ be the minimal QMS on $\mathfrak{B}(\mathsf{L}^2(Z))$ satisfying (\ref{eq-5-4}),
where $G= -\mathrm{i}H_f-\frac{1}{2}S_w$ with $H_f$ being given by (\ref{eq-5-7}).
Then, on the decoherence-free subalgebra $\mathfrak{N}(\mathcal{T})$, $\mathcal{T}$ takes the following form
\begin{equation}\label{eq-5-10}
\mathcal{T}_t(X) = \mathrm{e}^{\mathrm{i}tH_f}X\mathrm{e}^{-\mathrm{i}tH_f},\quad X \in \mathfrak{N}(\mathcal{T}),\ t\geq 0.
\end{equation}
\end{proposition}

\begin{proof}
Let $L_{jk}=\sqrt{w(j,k)}\,\partial_j^*\partial_k$ for $j$, $k\geq 0$
and take $D=\mathrm{Dom}\, G$. Then, after a lengthy but easy check, we can find
that the operators $L_{jk}$, linear manifold $D$ and QMS $\mathcal{T}$ satisfy all conditions listed in Theorem 3.2 of \cite{dhahri}.
Thus, by that theorem, relation (\ref{eq-5-10}) holds.
\end{proof}

In what follows, we always assume that $f\colon \Gamma \rightarrow \mathbb{R}$ is a bounded function and $w$ is a regular transition kernel.
Recall that $S_w$ is the 2D-weighted number operator associated with $w$.
The following two propositions then characterize the decoherence-free subalgebra $\mathfrak{N}(\mathcal{T})$ in terms of operators $S_w$, $\partial_k$ and $\partial_k^*$.

\begin{proposition}\label{prop-5-7}
Let $\mathcal{T} =(\mathcal{T}_t)_{t\geq 0}$ be the minimal QMS on $\mathfrak{B}(\mathsf{L}^2(Z))$ satisfying (\ref{eq-5-4}),
where $G= -\mathrm{i}H_f-\frac{1}{2}S_w$ with $H_f$ being given by (\ref{eq-5-7}). Then,
for all $X \in \mathfrak{N}(\mathcal{T})$ and all $t\geq 0$, $\mathcal{T}_t(X)$ leaves $\mathrm{Dom}\, S_w$ invariant, and moreover it holds  on $\mathrm{Dom}\, S_w$ that
\begin{equation}\label{eq-5-11}
   S_w\mathcal{T}_t(X) = \mathcal{T}_t(X)S_w.
\end{equation}
In particular, for all $X \in \mathfrak{N}(\mathcal{T})$, $\mathrm{Dom}\, S_w$ is an invariant subspace of $X$, and it holds on $\mathrm{Dom}\, S_w$ that
\begin{equation}\label{eq-5-12}
  S_wX = XS_w.
\end{equation}
\end{proposition}

\begin{proof}
According to \cite{dhahri}, the decoherence-free subalgebra $\mathfrak{N}(\mathcal{T})$ coincides with the commutator of the following
operator family
\begin{equation}\label{eq-5-13}
  \mathfrak{D}(\mathcal{T})
  =\left\{\mathrm{e}^{-\mathrm{i}tH_f}L_{jk}\mathrm{e}^{\mathrm{i}tH_f},\, \mathrm{e}^{-\mathrm{i}tH_f}L_{jk}^*\mathrm{e}^{\mathrm{i}tH_f}\mid j,\, k\geq 0,\, t\geq 0\right\},
\end{equation}
where $L_{jk}=\sqrt{w(j,k)}\,\partial_j^*\partial_k$. Let $X \in \mathfrak{N}(\mathcal{T})$ and $t\geq 0$. Then, by using the above relation, we have
\begin{equation*}
  \mathrm{e}^{-\mathrm{i}tH_f}L_{jk}\mathrm{e}^{\mathrm{i}tH_f} X = X \mathrm{e}^{-\mathrm{i}tH_f}L_{jk}\mathrm{e}^{\mathrm{i}tH_f},\quad j,\, k\geq 0.
\end{equation*}
which together with Proposition~\ref{prop-5-6} implies that
\begin{equation*}
  L_{jk}\mathcal{T}_t(X) = \mathcal{T}_t(X)L_{jk},\quad j,\, k\geq 0.
\end{equation*}
Similarly, we can get
\begin{equation*}
  L_{jk}^*\mathcal{T}_t(X) = \mathcal{T}_t(X)L_{jk}^*,\quad j,\, k\geq 0.
\end{equation*}
Thus
\begin{equation}\label{eq-5-14}
   L_{jk}^*L_{jk}\mathcal{T}_t(X) = \mathcal{T}_t(X)L_{jk}^*L_{jk},\quad j,\, k\geq 0,
\end{equation}
which, together with the notation $L_{jk}=\sqrt{w(j,k)}\,\partial_j^*\partial_k$ as well as Definition~\ref{def-4-1}, implies that
$\mathcal{T}_t(X)$ leaves $\mathrm{Dom}\, S_w$ invariant and $S_w\mathcal{T}_t(X) = \mathcal{T}_t(X)S_w$ on $\mathrm{Dom}\, S_w$.
\end{proof}

\begin{proposition}\label{prop-5-8}
Let $\mathcal{T} =(\mathcal{T}_t)_{t\geq 0}$ be the minimal QMS on $\mathfrak{B}(\mathsf{L}^2(Z))$ satisfying (\ref{eq-5-4}),
where $G= -\mathrm{i}H_f-\frac{1}{2}S_w$ with $H_f$ being given by (\ref{eq-5-7}). Then,
for all $X \in \mathfrak{N}(\mathcal{T})$, it holds that
\begin{equation}\label{eq-5-15}
  \partial_k\mathcal{T}_t(X)\partial_k=0,\quad \partial_k^*\mathcal{T}_t(X)\partial_k^*=0,\quad k\geq 0,\ t\geq 0.
\end{equation}
In particular, for all $X \in \mathfrak{N}(\mathcal{T})$, it holds that
\begin{equation}\label{eq-5-16}
  \partial_kX\partial_k=0,\quad \partial_k^*X\partial_k^*=0,\quad k\geq 0.
\end{equation}
\end{proposition}

\begin{proof}
Let $X \in \mathfrak{N}(\mathcal{T})$ and $t\geq 0$. By (\ref{eq-5-14}), we have
\begin{equation*}
  L_{kk}^*L_{kk}\mathcal{T}_t(X) = \mathcal{T}_t(X)L_{kk}^*L_{kk},\quad  k\geq 0,
\end{equation*}
which, together with the notation $L_{jk}=\sqrt{w(j,k)}\,\partial_j^*\partial_k$,  Lemma~\ref{lem-3-3} as well as the regularity of transition kernel $w$,
yields
\begin{equation*}
  \partial_k^*\partial_k\mathcal{T}_t(X) = \mathcal{T}_t(X)\partial_k^*\partial_k,\quad k\geq 0,
\end{equation*}
which together with Lemma~\ref{lem-3-2} gives
\begin{equation}\label{eq-5-17}
  \partial_k^*\partial_k\mathcal{T}_t(X)\partial_k =0,\quad
  \partial_k^*\mathcal{T}_t(X)\partial_k^*\partial_k =0, \quad k\geq 0.
\end{equation}
Now let $k\geq 0$. Then, for all $\sigma \in \Gamma$, by Lemma~\ref{lem-3-1} we find
\begin{equation*}
  \partial_k^*\mathcal{T}_t(X)\partial_k^*Z_{\sigma}=
  \left\{
    \begin{array}{ll}
      \partial_k^*\mathcal{T}_t(X)(\partial_k^*Z_{\sigma}) = 0, & \hbox{$k\in \sigma$;} \\
      \partial_k^*\mathcal{T}_t(X)\partial_k^*\partial_kZ_{\sigma\cup k}=0, & \hbox{$k\notin \sigma$,}
    \end{array}
  \right.
\end{equation*}
which, together with the fact that $\{Z_{\sigma} \mid \sigma \in \Gamma\}$ is an ONB, implies that
$\partial_k^*\mathcal{T}_t(X)\partial_k^*=0$.
Similarly,  for all $\sigma$, $\tau \in \Gamma$, we have
\begin{equation*}
  \langle Z_{\sigma}, \partial_k\mathcal{T}_t(X)\partial_kZ_{\tau}\rangle=
   \left\{
     \begin{array}{ll}
       \langle \partial_k^*Z_{\sigma}, \mathcal{T}_t(X)\partial_kZ_{\tau}\rangle=0, & \hbox{$k\in \sigma$;} \\
       \langle Z_{\sigma\cup k}, \partial_k^*\partial_k\mathcal{T}_t(X)\partial_kZ_{\tau}\rangle=0, & \hbox{$k\notin \sigma$,}
     \end{array}
   \right.
\end{equation*}
which implies that $\partial_k\mathcal{T}_t(X)\partial_k=0$. It then follows from the arbitrariness of $t\geq0$ and $k\geq0$ that (\ref{eq-5-15}) holds.
\end{proof}

\section*{Acknowledgement}

The authors are extremely grateful to the referees for their valuable comments and suggestions on improvement
of the first version of the present paper.
This work is supported by National Natural Science Foundation of China (Grant No. 11461061, 11861057).

\end{document}